\documentclass[11pt]{article}
\usepackage{enumerate}
\usepackage{lineno,hyperref,extarrows}
\usepackage{amsfonts}
\usepackage{amssymb}
\usepackage{mathrsfs}
\usepackage{srcltx}
\usepackage{verbatim}
\usepackage{amsmath}
\usepackage{amsthm}
\usepackage{amstext}
\usepackage{amsopn}
\usepackage{graphicx}
\modulolinenumbers[4]
\allowdisplaybreaks[4]

\newtheorem{theorem}{Theorem}[section]
\newtheorem{lemma}[theorem]{Lemma}

\newtheorem{corollary}[theorem]{Corollary}
\newtheorem{remark}[theorem]{Remark}

\theoremstyle{definition}

\numberwithin{equation}{section}

\newcommand{\ba}{\begin{array}}
\newcommand{\ea}{\end{array}}
\newcommand{\f}{\frac}

\newcommand{\la}{\lambda}

\newcommand{\R}{\mathbb{R}}

\newcommand{\N}{\mathbb{N}}

\begin{document}
\date{}

 \title{\bf
 Normalized solutions for nonhomogeneous  Chern-Simons-Schr\"odinger equations \\ with critical exponential growth
 \footnote{E-mail: {\tt clweimath@163.com} (C.L. Wei),  {\tt mathsitongchen@mail.csu.edu.cn} (S.T. Chen), {\tt zhouxinaomath@163.com} (X.A. Zhou)  .  }}
\author{ Chenlu Wei,\ Sitong Chen\footnote{Corresponding Author.},  \ Xin'ao Zhou\\
{\small \it School of Mathematics and Statistics, HNP-LAMA, Central South University,}\\
{\small \it  Changsha, Hunan 410083, P.R.China}}

\maketitle
\begin{abstract}
This paper investigates the existence of normalized solutions for the following Chern-Simons-Schr\"odinger equation:
\begin{equation*}
 \left\{
   \begin{array}{ll}
     -\Delta u+\lambda u+\left(\frac{h^{2}(\vert x\vert)}{\vert x\vert^{2}}+\int_{\vert x\vert}^{\infty}\frac{h(s)}{s}u^{2}(s)\mathrm{d}s\right)u
      =\left(e^{u^2}-1\right)u+g(x), & x\in \R^2, \\
   u\in H_r^1(\R^2),\  \int_{\R^2}u^2\mathrm{d}x=c, \\
   \end{array}
 \right.
 \end{equation*}
 where $c>0$, $\lambda\in \R$ acts as a Lagrange multiplier and   $g\in \mathcal {C}(\mathbb{R}^2,[0,+\infty))$ satisfies suitable assumptions.
In addition to  the loss  of compactness caused by the nonlinearity with critical exponential growth, the intricate interactions among it, the nonlocal term, and the nonhomogeneous term significantly affect the geometric structure of the constrained functional, thereby making this research particularly challenging.
By specifying  explicit conditions on $c$, we subtly
 establish a structure of local minima of the constrained functional.
Based on the structure, we employ new  analytical techniques to  prove the existence of two solutions: one being a
local minimizer and one of mountain-pass type.
Our results are entirely new, even for the Schr\"odinger equation that is when nonlocal terms are absent. We believe our methods may be adapted and modified to deal with more constrained problems with nonhomogeneous perturbation.


\end{abstract}











{\section{Introduction}}
 \setcounter{equation}{0}
 In this paper, we study  the following Chern-Simons-Schr\"odinger equation with a nonhomogeneous perturbation:
\begin{equation}\label{Pa1}
 \left\{
   \begin{array}{ll}
     -\Delta u+\lambda u+\left(\frac{h^{2}(\vert x\vert)}{\vert x\vert^{2}}+\int_{\vert x\vert}^{\infty}\frac{h(s)}{s}u^{2}(s)\mathrm{d}s\right)u
      =\left(e^{u^2}-1\right)u+g(x), & x\in \R^2, \\
   u\in H_r^1(\R^2),\  \int_{\R^2}u^2\mathrm{d}x=c, \\
   \end{array}
 \right.
 \end{equation}
 where $h(s)=\int_{0}^{s}\frac{l}{2}u^{2}(l)\mathrm{d}l$,
$c>0$ is  a given real number, $\lambda\in \R$ is unknown and arises as a Lagrange multiplier which depends on the solution $u\in H_r^1(\R^2)$,
$g:\mathbb{R}^2 \rightarrow \mathbb{R}$ satisfies the following basic assumption:
\begin{itemize}
		\item[(G1)]  $g\in \mathcal {C}(\mathbb{R}^2,[0,+\infty)) \cap L^{\frac{4}{3}}(\mathbb{R}^2)$, $g(x)=g(|x|)$ and $g(x)\not\equiv 0$ for all $x=(x_1,x_2) \in \mathbb{R}^2$.
	\end{itemize}
As in \cite{AY,dMR1}, we say that $f$ enjoys {\it critical  exponential growth} (see Adimurthi -Yadava \cite{AY}, as well as de Figueiredo-Miyagaki-Ruf \cite{dMR1}) if
 $f\in \mathcal{C}(\R, \R)$ satisfies
 \begin{equation}\label{Cr1}
  \lim_{|t|\to +\infty}\frac{|f(t)|}{e^{\alpha t^2}}=
   \begin{cases}
     0, \;\;&  \mbox {for all } \alpha>\alpha_0,\\
     +\infty,&  \mbox {for all } \alpha<\alpha_0,
 \end{cases}\ \ \ \ \hbox{for some constant}\  \alpha_0>0,
 \end{equation}
which is motivated by  Trudinger-Moser inequality developed by Cao in \cite{Cao},  see Lemma \ref{lem 2.3}  for more details. Thus, $ \left(e^{u^2}-1\right)u $ admits critical  exponential growth. Here, $H^1_r(\R^2)$ denotes the set of radially symmetric functions in  $H^1(\R^2)$.
\par
The  solutions of \eqref{Pa1} can be obtained as critical points of the energy functional $\Phi:H^1_r(\R^2)\rightarrow \R$ defined by
\begin{equation}\label{fai}
		\Phi(u):= \frac{1}{2}\int_{\R^2}\left[|\nabla u|^2+\frac{u^{2}}{\vert x\vert^{2}}\left(\int_{0}^{\vert x\vert}\frac{s}{2}u^{2}(s)\mathrm{d}s\right)^{2}\right]\mathrm{d}x
		-\frac{1}{2}\int_{\R^2}\left(e^{u^2}-1-u^2\right)\mathrm{d}x
-\int_{\R^2}g(x)u\mathrm{d}x
	\end{equation}
under the constraint
	\begin{equation}\label{Sa}
		\mathcal{S}_c:=\left\{u\in H^1_{r}(\R^2):\|u\|_2^2=c>0\right\}.
	\end{equation}
Moreover,   we know  that $\Phi \in \mathcal{C}^1(H^1_r(\R^2),\R)$, see Section 2 for more details.
\par
The first equation of \eqref{Pa1} is derived from  the study of standing waves for the following time dependent Chern-Simons-Schr\"odinger system  with the gauge field in $\R^2$:
 \begin{align}\label{css1}
 \left\{
\begin{aligned}
&iD_{0}\phi+(D_{1}D_{1}+D_{2}D_{2})\phi=f(\phi),\\
&\partial_{0}A_{1}-\partial_{1}A_{0}=-\text{Im}(\bar{\phi}D_{2}\phi),\\
&\partial_{0}A_{2}-\partial_{2}A_{0}=\text{Im}(\bar{\phi}D_{1}\phi),\\
&\partial_{1}A_{2}-\partial_{2}A_{1}=-1/2\vert \phi\vert^{2},
\end{aligned}
\right.
\end{align}
where $i$ denotes the imaginary unit, $\partial_{0}=\frac{\partial}{\partial t}$, $\partial_{1}=\frac{\partial}{\partial x_{1}}$, $\partial_{2}=\frac{\partial}{\partial x_{2}}$ for $(t, x_{1}, x_{2})\in \mathbb{R}\times \mathbb{R}^{2} $, $\phi:\mathbb{R}\times \mathbb{R}^{2}\rightarrow \mathbb{C}$ is the complex
scalar field, $A_{j}:\mathbb{R}\times \mathbb{R}^{2}\rightarrow \mathbb{R}$ is the gauge field, $D_{j}=\partial_{j}+iA_{j}$ is the covariant derivative for $j=0, 1, 2$ and the function $f$ denotes the nonlinearity.
In particular, we consider standing wave solutions of the following form
\begin{align}\label{SW}
		&\phi(t, x)=u(\vert x\vert)e^{i\lambda t}, \quad\quad\quad A_{0}(x, t)=k(\vert x\vert),\nonumber\\
		&A_{1}(x, t)=\frac{x_{2}}{\vert x\vert}h(\vert x\vert),\quad\quad\quad A_{2}(x, t)=-\frac{x_{1}}{\vert x\vert}h(\vert x\vert),
\end{align}
where $\lambda>0$ denotes a frequency and $u$, $k$ and $h$ represent real-valued functions on the interval $[0, +\infty)$ with the condition that $h(0)=0$.
Furthermore, if  the Coulomb gauge condition  $\partial_{1}A_{1}+\partial_{2}A_{2}=0$ is satisfied, one can easily conclude that the following non-local equation:
	\begin{equation}\label{PG1}
		-\Delta u+\lambda u+\left(\xi+\int_{\vert x\vert}^{\infty}\frac{h(s)}{s}u^{2}(s)ds\right)u
		+\frac{h^{2}(\vert x\vert)}{\vert x\vert^{2}}u
		=f(x,u),  \ \  x\in \R^2,
	\end{equation}
	where $h(s)=\int_{0}^{s}\frac{l}{2}u^{2}(l)\mathrm{d}l$, $\xi\in \R$ is a constant and $A_0=\xi+\int_{\vert x\vert}^{\infty}\frac{h(s)}{s}u^{2}(s)ds$.
This is a nonrelativistic quantum model which can be used to describe the behavior
 for a large number of particles in Chern-Simons gauge fields.
 The characteristics of this model play a crucial role in studying fractional quantum Hall effect, high-temperature supconductors, as well as Aharovnov-Bohm scattering phenomena, see \cite{Ja1,Ja2} for more physical information.
 \par	
	When studying equation \eqref{PG1},
	there exist two distinct options regarding the frequency $\lambda$,
	leading to two different research fields. Without loss of generality, we set $\xi=0$ in \eqref{PG1} following previous works such as \cite{BHS}.
	If the frequency $\lambda$ is fixed,
any solutions of  \eqref{PG1} can be viewed as critical points of the energy functional defined by:
\begin{align*}
		\Phi_{\lambda}(u)
		= \frac{1}{2}\int_{\R^2}\left[|\nabla u|^2+\lambda u^2+\frac{u^{2}}{\vert x\vert^{2}}\left(\int_{0}^{\vert x\vert}\frac{s}{2}u^{2}(s)\right)^{2}\right]\mathrm{d}x
		-\int_{\R^2}F(x,u)\mathrm{d}x, \ \ \ \ \forall\ u\in H^1_r(\R^2)
	\end{align*}
	with $F(u)=\int_0^uf(t)\mathrm{d}t$.
After the pioneering work  of  Byeon-Huh-Seok \cite{BHS} in this direction, equation \eqref{PG1} began to receive a great deal of  attention.
Especially, when considering
 equation \eqref{PG1} with nonhomogeneous perturbation, we are only aware of the papers \cite{JF-JMAA,Zhs}.
To be precise, Yuan-Weng-Zhou \cite{Zhs} obtained the existence and multiplicity of solutions for \eqref{PG1}  when $f(x,u)=|u|^{p-2}u+g(x)$ with $p>2$;
Ji-Fang \cite{JF-JMAA} demonstrated that equation \eqref{PG1} with small perturbations and critical exponential growth has at least two positive solutions.
For more researches for \eqref{PG1} with general nonlinearities, we refer the readers to \cite{CTY,dP1,DPS, HHS,Huh,HHJ,LLi,LOZ,PR2,PR1}.
\par
	Alternatively, it is of great interest to investigate solutions to \eqref{PG1} with   $\lambda\in \R$ as an unknown parameter.
In this situation, the corresponding solutions of \eqref{PG1} possess a
	prescribed $L^2$-norm and $\lambda\in \R$ acts as a Lagrange multiplier with respect to the constraint  $\mathcal{S}_c$.
Specifically, for any given  $c>0$, our purpose is to investigate solutions to  \eqref{PG1} under the $L^2$-norm constraint $\mathcal{S}_c$.
	Such solutions are usually referred to as   \emph{normalized solutions} for \eqref{PG1},
	which correspond formally to critical points of the functional $\Phi: H^1_r(\R^2) \to \R$ defined by
	\begin{align*}
		\Phi_{0}(u)
		= \frac{1}{2}\int_{\R^2}\left[|\nabla u|^2+\frac{u^{2}}{\vert x\vert^{2}}\left(\int_{0}^{\vert x\vert}\frac{s}{2}u^{2}(s)\mathrm{d}s\right)^{2}\right]\mathrm{d}x
		-\int_{\R^2}F(x,u)\mathrm{d}x
	\end{align*}
	restricted  on $\mathcal{S}_c$. From a physics standpoint, this trouble holds significant meaning,
since it  not only conserves  the $L^2$-norm of the solution over time when solving the Cauchy problem
for \eqref{PG1},
but also  offers valuable insights into the dynamical properties, such as orbital stability and instability of solutions for  \eqref{PG1}, see \cite{BC,CL1} and references therein.
\par
During the past decades, the existence of normalized solutions to equation \eqref{PG1} has been investigated  increasingly.
In particular, in the study of \eqref{PG1},
$L^2$-critical exponent ($p^*=2+\frac{4}{N}$) plays an important role in determining the existence of a lower bound for the energy functional $ \Phi_{0}(u)$  on   $\mathcal{S}_c$, as discussed in \cite{So-JDE}.
Specifically,  if $f(u)$ in \eqref{PG1} grows slower than $|u|^{p^*-2}u$ at infinity,
$ \Phi_{0}(u)$ is always bounded from below on $\mathcal{S}_c$  and corresponding normalized solutions can be obtained by  exploring a global minimum; this is referred to as an $L^2$-subcritical case.
Otherwise,
$ \Phi_{0}(u)$ is always unbounded from below on $\mathcal{S}_c$; this is referred to as an $L^2$-supcritical case and it seems impossible to
derive a global minimum.
For the results concerning equation \eqref{PG1} with power nonlinearity $f(u)=|u|^{p-2}u$,
 we  refer readers to  \cite{BHS,CX-AASFM,gtx,LLuo,Yuan} for further details.
 Moreover, equation \eqref{PG1} with combined nonlinearities
 $ f(u)=\mu|u|^{p-2}u+|u|^{q-2}u\ \hbox{where}\  2<p,q<+\infty$  has  also been investigated,
see \cite{YCS} for example. More precisely,  Yao-Chen-Sun \cite{YCS}  obtained   two  normalized  solutions $\bar{u}_c^{\pm}$ which satisfy $\Phi_0(\bar{u}_c^+)<0<\Phi_0(\bar{u}_c^-)$ when $2<p<4<q$  and $\mu=1$.
\par
Furthermore, remarkable results for the following Chern-Simons-Schr\"odinger system with critical exponential growth have emerged
 in recent years,
 \begin{align}\label{cssf}
 \left\{
\begin{aligned}
&-\Delta u+\lambda u+A_0u+A_1^2u+A_2^2u
     =f(u), & x\in \R^2, \\
&\partial_{1}A_{0}=A_2u^2, \ \partial_{2}A_{0}=-A_1u^2, \\
&\partial_{1}A_{2}-\partial_{2}A_{1}=-1/2|u|^2, \ \partial_{1}A_{1}+\partial_{2}A_{2}=0,\\
&\int_{\R^2}u^2\mathrm{d}x=c.
\end{aligned}
\right.
\end{align}
Yao-Chen-Sun \cite{YCS-JGA}  obtained the existence of normalized  solutions to equation  \eqref{cssf} where $f$ satisfies \eqref{Cr1} with $\alpha_0=4\pi$,
	and $|x|^2u$ is added to \eqref{Pa1}.
 When $f(u)$ is replaced by $a(|x|)f(u)$, Yuan-Tang-Chen \cite{Yuan1} overcame the additional difficulties arising from
the non-constant potential $a(|x|)$ and  establish the existence of normalized solutions to system \eqref{cssf};
Gao-Tan \cite{GaoTan} and Huang-Feng-Chen \cite{jmaaH} further investigated normalized solutions of system \eqref{cssf} by weakening  the conditions in \cite{YCS-JGA,Yuan1} and Huang-Feng-Chen \cite{jmaaH} additionally  established the existence of a ground state solution to  \eqref{cssf} in $H^1(\R^2)$. Recently, Wei-Wen \cite{wei} considered the  system \eqref{cssf} with combined nonlinearities of critical exponent and an $L^2$-critical perturbation and prove the existence of normalized solutions by establishing the specific value range of $c$.
\par
Inspired by these works, especially \cite{JF-JMAA,Zhs} which considered Chern-Simons-Schr\"odinger equation \eqref{PG1}  with nonhomogeneous perturbation,  a natural question arises in the context of $L^2$-constraint:
{\bf{
\begin{itemize}
\item[(Q)] What will happen if the nonlinearity in \eqref{PG1} is $h(u)+g(x)$,
				especially where $h(u)$ is of critical exponential growth?
\end{itemize}
	}}
\par
It appears that the study of normalized solutions  for the nonhomogeneous Chern-Simons-Schr\"odinger equations
with critical exponential growth  has not been explored  in the existing literature.
As will be seen,
with the absence of gauge fields, equation \eqref{PG1}   reduces to
the following planar scenario  ($N=2$) of the Schr\"odinger equation:
\begin{equation}\label{PS}
		\left\{
		\begin{array}{ll}
			-\Delta u+\lambda u=f(x,u), & x\in \R^N, \ N\ge 2, \\
			\int_{\R^N}u^2\mathrm{d}x=c>0. \\
		\end{array}
		\right.
	\end{equation}
The first contribution  for equation \eqref{PS}   with nonhomogeneous perturbation  was recently made by
Chen-Zou \cite{ZWM}, which primarily  focused on cases where $f(x,u)=|u|^{p-2}+g(x)$.
Specifically, Chen-Zou obtained a global minimizer for $2<p<2+\frac{4}{N}$ and established   the existence of a mountain pass solution for $2+\frac{4}{N}<p<2^*:=\left\{\small
	\begin{array}{ll}
		2N/(N-2),  & N\ge 3,\\
		+\infty,& N=1,2.\\
	\end{array}
	\right.$
It  is not difficult to find that
the type of normalized solutions strictly depends on the order $p$ of power function.
Particuarly, this related result further motivates our research.

\par
In this paper, the main goal is to give an definitive answer to this question and investigate the existence of normalized solutions to \eqref{Pa1}.
In  contrast to previous works, especially system \eqref{cssf}  with  critical exponential growth
and \cite{ZWM} investigating equation \eqref{PS}  with nonhomogeneous perturbation,  additional difficulties arise in the study of
\eqref{Pa1}, which can be summarized as follows.
On one hand,  compared with the system \eqref{cssf},
the presence of perturbation $g$  strongly affects
 the geometric structure of the constrained functional.
Specifically, even if
a solution of mountain-pass type
can be identified, the method of constructing this solution is fundamentally different from that of \cite{GaoTan,jmaaH}.
On the other hand, unlike the research conducted in \cite{ZWM}  considering $L^2$-subcritical case and $L^2$-supercritical case,  the critical exponential
growth case of equation \eqref{Pa1} poses extra challenges in terms of the compactness analysis.
Additionally, the interplay between the added non-local term
and the nonhomogeneous nonlinearity  in \eqref{Pa1} is more delicate,  which profoundly influences the selection of approaches for seeking constrained critical points.
These challenges enforce the implementation of new ideas and mathematical methods  (see Remark \ref{rem 1.2})  to capture normalized
solutions of \eqref{Pa1}.
Finally,  we prove
the existence of two solutions, one being a local minimizer with negative energy and one of mountain-pass type with positive energy.
\par
To better illustrate our methods, we present a concrete  model with critical exponential growth of $h(u)$,  namely, $f(x,u)=  \left(e^{u^2}-1\right)u+g(x)$.
Before stating our conclusion, we initially introduce our hypothesis about $g$, which satisfies the following assumptions in addition to {\rm(G1)}:

\begin{itemize}
		\item[(G2)] $\nabla g(x)\cdot x \in  L^{\frac{4}{3}}(\mathbb{R}^2)$ and $2g(x)+\nabla g(x)\cdot x\ge 0$ for all $x \in \mathbb{R}^2$;
	\end{itemize}
\begin{itemize}
		\item[(G3)] $ t \shortmid \rightarrow g(tx)$ is nonincreasing on $(0,\infty)$ for every $x \in \R^2$;
	\end{itemize}
\begin{itemize}
		\item[(G4)] there exists a constant  $\theta \in (0,1)$ such that for every $x \in \R^2$,
$$g(x)-g(\theta x) \leq (1-\theta)\nabla g(x)\cdot x.$$
	\end{itemize}
 We define now the $L^2$-Pohozaev functional
\begin{align}\label{pa4}
		\mathcal{P}(u)
		& : = \|\nabla u\|_2^2+\int_{\R^2}\frac{u^{2}}{\vert x\vert^{2}}\left(\int_{0}^{\vert x\vert}\frac{s}{2}u^{2}(s)\mathrm{d}s\right)^{2}\mathrm{d}x
		-\int_{\R^2}\left[\left(u^2-1\right)e^{u^2}+1\right]\mathrm{d}x\nonumber\\
		&  \ \ \ \ +\int_{\R^2}\left[g(x)+\nabla g(x)\cdot x\right]u\mathrm{d}x , \  \ \forall\ u\in H^1_r(\R^2).
	\end{align}
As one will see in Lemma \eqref{Lep},  any solution to \eqref{Pa1}  satisfies the $L^2$-Pohozaev identity $\mathcal{P}(u)=0$.

 \par
 Our  results are as follows:

 \begin{theorem}\label{thm 1.1} Assume that {\rm(G1)}-{\rm(G2)} hold. Then there exist $c_0>0$
		and $s_0\in \left(0,\f{\pi}{3}\right]$ such that, for any $c\in (0,c_0)$, \eqref{Pa1}  has a  couple solution  $(\bar{u},\bar{\lambda}_c)\in \mathcal{S}_c \times \R$  with some $\bar{\lambda}_c>0$ and satisfying
			$$
			\Phi(\bar{u})=m(c):=\inf \left\{ \Phi(u):    u\in \mathcal{S}_c, \ \|\nabla u\|_2^2 < s_0\right\}.
			$$
	\end{theorem}

	\begin{theorem}\label{thm 1.2}
Assume that {\rm(G1)}-{\rm(G4)} hold. Then there exist $c_0>0$
		and $s_0 \in \left(0,\f{\pi}{3}\right]$ such that, for any $c\in (0,c_0)$,
\eqref{Pa1}  has a  second couple solution  $(\hat{u},\hat{\lambda}_c)\in \mathcal{S}_c \times \R$  with some $\hat{\lambda}_c>0$ and satisfying
			\begin{equation}\label{M24}
				0<\Phi(\hat{u})< m(c)+2\pi.
			\end{equation}
	\end{theorem}

\begin{remark}\label{rem 1.1}
There are indeed functions which satisfy {\rm(G1)}-{\rm(G4)}. For example:
\begin{equation}\label{wnt22}
		g(x)=
		\begin{cases}
			1/2, \ \ & 0\le |x|\le 1;\\
			(|x|^2+1)^{-1}, \ \ &  |x|\geq 1 .
		\end{cases}
	\end{equation}
\end{remark}
\par

\begin{remark}\label{rem 1.2}
i)\ Our approach to establishing the geometry of local minima in Theorem {\rm\ref{thm 1.1}}, particularly when verifying their sign,
is fundamentally different  from that of {\rm\cite{YCS}}.
Specifically,  the presence of an $L^2$-subcritical term means that Yao-Chen-Sun  {\rm\cite{YCS}} can easily deduce $\lim_{t \to 0^{+}}\Phi_0(tu(tx))=0^{-}$
where $tu(tx)$ is a dilation preserving the $L^2$-norm.
However, due to the presence of $g$, our  situation is more complicated as it remains unclear whether $\lim_{t \rightarrow 0^{+}}\Phi(tu(tx)) = 0^{-}$
where $\Phi(tu(tx))$ is given as follows:
\begin{align}\label{hud1111}
			 \Phi(tu(tx))=\f{t^2}{2}\left[\|\nabla u\|_2^2+B(u)\right]
			-\f{1}{2t^{2}}\int_{\R^2}\left(e^{t^2u^2}-1-t^2u^2\right)\mathrm{d}x-\f{1}{t}\int_{\R^2}g\left(\frac{x}{t}\right)u\mathrm{d}x.
		\end{align}
In contrast, by carefully selecting  suitable functions $\tilde{u}$, we derive the pivotal 
 inequality $   \Phi(t\tilde{u}(tx))\le \bar{H}_{\tilde{u}}(t)  $  and  $\bar{H}_{\tilde{u}}(t) \rightarrow   0^{-}    $  as $t \rightarrow 0^{+}$
in a tactfully indirect manner
which enables the determination of the sign of local minima, as detailed in Lemma {\rm\ref{lem 3.3}}.
Moreover, the  idea can be applied to treat a wide range of nonlinearities involving  nonhomogeneous perturbation, such as,
\begin{equation*}
f(x,u)=g(x)+h(u),
\end{equation*}
  where  $h(u)$ satisfies \eqref{Cr1} or $h(u) \sim |u|^{p-2}u,\ \ 2<p<+\infty$.
\par
ii)\ In the proof of Theorem {\rm\ref{thm 1.2}}, to restore the lack of compactness caused by the exponential term, the key issue is to
establish  a precise upper estimate of the  energy level $M(c)$, namely,
	\begin{equation}\label{MMCm}
		M(c)<m(c)+2\pi.
	\end{equation}
Inspired by recent works {\rm\cite{ctmz,ctcv}}, in order to derive this inequality \eqref{MMCm}, we consider a superposition of Moser-type functions and the previously obtained first solution, while ensuring that the resulting function remains constrained to $\mathcal{S}_c$ through appropriate technical modifications.
However, due to the additional difficulties arising from the nonhomogeneous perturbation and the nonlocal term,  one  needs to develop delicate analyses to control the energy. For more details, see \eqref{BWt}, \eqref{in} and Lemma {\rm\ref{lem 3.14}}.

iii)\
In  this paper,   the existence of normalized solutions   to  nonhomogeneous Chern-Simons-Schr\"odinger equation       with critical exponential growth is first studied.
Our results are entirely new, even for the Schr\"odinger equation that is when nonlocal terms are absent. We believe our methods may be adapted and modified to deal with more constrained problems with nonhomogeneous perturbation.
\end{remark}

\bigskip
  \par
 The paper is organized as follows. Section 2 is devoted to some  preliminary results, which will  be used in the rest paper. In Section 3, we consider the existence of a local minima for $\Phi$ on $\mathcal{S}_c\cap A_{s_0}$, and give the proof of Theorem \ref{thm 1.1}. In Section 4, we study the
existence of a critical point of mountain-pass type for $\Phi$ on $\mathcal{S}_c$, and complete the proof of
 Theorem \ref{thm 1.2}.

 \bigskip
 \par
 Throughout the paper, we make use of the following notations:

  $\bullet$ \  $H^1(\R^2)$ denotes the usual Sobolev space equipped with the inner product and norm
   $$
   (u,v)=\int_{\R^2}(\nabla u\cdot \nabla v +uv)\mathrm{d}x, \ \
 \|u\|=(u,u)^{1/2}, \ \ \forall\ u, v\in H^1(\R^2);
 $$

 \par
  $\bullet$ \ $H_{r}^1(\R^2)$ denotes the space of spherically symmetric functions belonging to $H^1(\R^2)$:
  $$
   H_{r}^1(\R^2):=\{u\in H^1(\R^2)\ \big|\ u(x)=u(|x|)\ \hbox{a.e. in } \R^2\};
  $$

     $\bullet$ \ $L^s(\R^2) (1\le s< \infty)$  denotes the Lebesgue space with the norm $\|u\|_s
 =\left(\int_{\R^2}|u|^s\mathrm{d}x\right)^{1/s}$;

 \par
     $\bullet$ \  For any $u\in H^1(\R^2)\setminus \{0\}$, $u_t(x):=u(tx)$ for $t>0$;

 \par
     $\bullet$ \ For any $x\in \R^2$ and $r>0$, $B_r(x):=\{y\in \R^2: |y-x|<r \}$
     and $B_r=B_r(0)$;

 \par
     $\bullet$ \ $C_1, C_2,\cdots$ denote positive constants possibly different in different places, which are dependent on $c>0$.

{\section{ Preliminary results}}
	\setcounter{equation}{0}
	
	In this section, we give some preliminary results which will be often used throughout the rest of the paper.
	 \begin{lemma}\label{lemGN} {\rm \cite{We}(Gagliardo-Nirenberg inequality)} 
 Let $s> 2$ and $u\in H^1(\R^2)$. Then there exists a sharp constant $\mathcal{C}_s>0$ such that
 \begin{equation}\label{GNs}
		\|u\|_s\le \mathcal{C}_s\|u\|_2^{\frac{2}{s}}\|\nabla u\|_2^{\frac{s-2}{s}} \ \ \ \ \mbox{for}\ \ s>2.
	\end{equation}
 \end{lemma}
	
	\begin{lemma}\label{lem 2.2}{\rm\cite{ZTC}}
		For any $u\in H^1_{r}(\R^2)$ and $r_0>0$, there holds
		\begin{equation}\label{lm01}
			|u(x)|\le \frac{1}{\sqrt{\pi |x|}}\|u\|_2^{1/2}\|\nabla u\|_2^{1/2},
			\ \ \ \ \forall \ |x|\ge r_0.
		\end{equation}
	\end{lemma}

For any $u\in H_r^1(\R^2)$, we define
\begin{equation}\label{Bu}
		B(u):= \int_{\R^2}\frac{u^{2}}{|x|^{2}}\left(\int_{0}^{|x|}\frac{s}{2}u^{2}(s)\mathrm{d}s\right)^{2}\mathrm{d}x
		=\int_{\R^2}\frac{u^{2}}{4|x|^{2}}\left(\frac{1}{2\pi}\int_{B_{|x|}}u^{2}(y)\mathrm{d}y\right)^{2}\mathrm{d}x.
\end{equation}
In view of \cite[Lemma 2.3]{LL} and \cite[Proposition 2.4]{BHS}, we have the following important estimates on $B(u)$:

	\begin{lemma}\label{lemB.1} The following inequalities hold:
		\begin{equation}\label{Bu1}
			B(u) \le \frac{1}{16\pi}\|u\|_2^2\|u\|_4^4,
			\ \ \ \ \forall\ u\in H_r^1(\R^2)
		\end{equation}
		and
		\begin{equation}\label{Bu2}
			\|u\|_4^4\le 4 \|\nabla u\|_2 B(u)^{1/2},
			\ \ \ \ \forall\ u\in H_r^1(\R^2).
		\end{equation}
	\end{lemma}

	Moreover, there is  the following compactness lemma we use later:
	
	\begin{lemma}\label{lem2.1} {\rm(\cite[Lemma 3.2]{BHS})}
		Suppose that a sequence  $\{u_{n}\}$ converges weakly to a function $u$ in $H^{1}_{r}(\mathbb{R}^{2})$ as $n\rightarrow \infty$. Then
		for each $\varphi\in H^{1}_{r}(\mathbb{R}^{2})$, $B(u_{n})$, $\left\langle B'(u_{n}),\varphi\right\rangle$
		and $\left\langle B'(u_{n}),u_{n}\right\rangle$ converge up to a subsequence to $B(u)$,
		$\left\langle B'(u),\varphi\right\rangle$ and $\left\langle B'(u),u\right\rangle$, respectively, as $n\rightarrow \infty$.
	\end{lemma}
	
	\begin{lemma}\label{lem 2.3}{\rm  \cite{AT,Cao,CST}}
   {\rm (i)} If $\alpha>0$ and $u\in H^1(\R^2)$, then
 $$
   \int_{\R^2}\left(e^{\alpha u^2}-1\right)\mathrm{d}x<\infty;
 $$
  {\rm (ii)} if $u\in H^1(\R^2)$, $\|\nabla u\|_2^2\le 1$, $\|u\|_2 \le \beta < \infty$ and  $\alpha<4\pi$,
  then there exists a constant $C(\alpha,\beta)$, which depends only on $\alpha$ and $\beta$, such that
 \begin{equation}\label{TM1}
   \int_{\R^2}\left(e^{\alpha{u^2}}-1\right)\mathrm{d}x\le C(\alpha,\beta).
 \end{equation}
 \end{lemma}

	By {\rm(G1)}, Lemmas \ref{lem2.1} and \ref{lem 2.3}, we have $\Phi\in \mathcal{C}^1(H_r^1(\R^2), \R)$,
	\begin{align}\label{Phd}
		\langle \Phi'(u), v \rangle
		&  =  \int_{{\R}^2}\nabla u\cdot\nabla v\mathrm{d}x+\int_{\mathbb{R}^{2}}\frac{u^{2}}{|x|^{2}}\left(\int_{0}^{|x|}\frac{s}{2}u^{2}(s)\mathrm{d}s\right)
		\left(\int_{0}^{|x|}su(s)v(s)\mathrm{d}s\right)\mathrm{d}x \nonumber\\
		& \ \ \ \ +\int_{\mathbb{R}^{2}}\frac{uv}{|x|^{2}}\left(\int_{0}^{|x|}\frac{s}{2}u^{2}(s)\mathrm{d}s\right)^2\mathrm{d}x
		-\int_{\R^2}g(x)v\mathrm{d}x
		-\int_{\R^2}\left(e^{u^2}-1\right)uv\mathrm{d}x, \nonumber\\
		&\hspace{9cm} \ \ \ \ \ \ \ \forall\ u,v\in H_r^1(\R^2),
	\end{align}
	and
	\begin{align}\label{Phd1}
		\langle \Phi'(u), u \rangle
		&  = \int_{\R^2}|\nabla u|^2\mathrm{d}x+3 B(u)-\int_{\R^2}\left(e^{u^2}-1\right)u^2\mathrm{d}x-\int_{\R^2}g(x)u\mathrm{d}x, \forall\ u\in H_r^1(\R^2).
	\end{align}
	Moreover, \eqref{pa4} and \eqref{Bu} give that
	\begin{equation}\label{Pu}
		\mathcal{P}(u) = \|\nabla u\|_2^2+ B(u)
		 -\int_{\R^2}\left[\left(u^2-1\right)e^{u^2}+1\right]\mathrm{d}x+\int_{\R^2}\left[g(x)+\nabla g(x)\cdot x\right]u\mathrm{d}x, \ \ \ \ \forall\ u\in H_r^1(\R^2).
	\end{equation}
	
	\begin{lemma}\label{lem 2.4}\rm(\cite[Lemma 2.4]{ctcv})
		If $u\in H^1(\R^2)$, then
		\begin{equation}\label{u2k}
			\int_{\R^2}|u|^{2k}\mathrm{d}x
			\le \f{2+2^{2k-1}(k-2)}{(k-2)\pi^{k-1}}\|u\|_2^{k}\|\nabla u\|_2^{k}
			+\f{k!}{2\pi^{k-1}}\|\nabla u\|_2^{2k}, \ \  \forall \ k\in \N, k>2,
		\end{equation}
		\begin{align}\label{D06}
			&      \int_{\R^2}\left(e^{u^2}-1-u^2\right)\mathrm{d}x \le  \f{1}{2}\|u\|_4^4+\f{\|\nabla u\|_2^6}{2\pi(\pi-\|\nabla u\|_2^2)}\nonumber\\
			&      \ \  +\f{2c^{\f{3}{2}}\|\nabla u\|_2^3}{\pi^2}\sum_{k=0}^{\infty}\f{4^{k+2}(k+1)+1}{(k+1)(k+3)!}
			\left(\f{\|u\|_2\|\nabla u\|_2}{\pi}\right)^k, \ \ \forall \ u\in \mathcal{S}_c, \|\nabla u\|_2^2<\pi
		\end{align}
		and
		\begin{align}\label{D08}
			&      \int_{\R^2}\left[\left(u^2-1\right)e^{u^2}+1\right]\mathrm{d}x
			\le \f{1}{2}\|u\|_4^4+\f{\left(2\pi-\|\nabla u\|_2^2\right)\|\nabla u\|_2^6}{2\pi\left(\pi-\|\nabla u\|_2^2\right)^2}\nonumber\\
			&      \ \  +\f{2c^{\f{3}{2}}\|\nabla u\|_2^3}{\pi^2}\sum_{k=0}^{\infty}\f{(k+2)\left[4^{k+2}(k+1)+1\right]}
			{(k+1)(k+3)!}\left(\f{\|u\|_2\|\nabla u\|_2}{\pi}\right)^k,\nonumber\\
			& \hspace{6cm}     \forall \ u\in \mathcal{S}_c, \|\nabla u\|_2^2<\pi.
		\end{align}
	\end{lemma}	
	\par
	
	In view of \cite[Proposition 2.3]{BHS}, we have the following lemma.
	
	\begin{lemma}\label{Lep}
		If $(u,\la_c)\in   H^1(\R^2)\times \R$ is a weak solution to \eqref{Pa1}, then $\mathcal{P}(u)=0$,
		where $\mathcal{P}$ is defined by \eqref{Pu}.
	\end{lemma}

	\par
	Finally, we present several critical point theorems on the manifold.
	Before this, we give some necessary definitions and notations.
	
	\par
	Let $H$ be a real Hilbert space whose norm and scalar product will be denoted
	respectively by $\|\cdot\|_H$ and $(\cdot, \cdot)_H$. Let $E$ be a real Banach space with norm $\|\cdot\|_E$. We
	assume throughout this section that
	\begin{equation}\label{EH}
		E\hookrightarrow H \hookrightarrow E^*
	\end{equation}
	with continuous injections, where $E^*$ is the dual space of $E$. Thus $H$ is identified with its dual space.
	We will always assume in the sequel that $E$ and $H$ are infinite dimensional spaces. We consider the manifold
	\begin{equation}\label{MD}
		M:=\{u\in E: \|u\|_H=1\}.
	\end{equation}
	$M$ is the trace of the unit sphere of $H$ in $E$ and it is, in general, unbounded. Throughout the paper, $M$
	will be endowed with the topology inherited from $E$. Moreover $M$ is a submanifold of $E$ of codimension
	1 and its tangent space at a given point $u\in M$ can be considered as a closed subspace of $E$ of codimension 1, namely
	\begin{equation}\label{TM}
		T_uM:=\{v\in E: (u,v)_H=0\}.
	\end{equation}
	
	\par
	We consider a functional $\varphi: E\rightarrow \R$ which is of class $\mathcal{C}^1$ on $E$.
	We denote by $\varphi|_{M}$ the trace of $\varphi$ on $M$. Then $\varphi|_{M}$  is a $\mathcal{C}^1$
	functional on $M$, and for any $u\in M$,
	\begin{equation}\label{DM}
		\langle\varphi|_{M}'(u), v\rangle=\langle\varphi'(u), v\rangle, \ \  \forall \ v\in T_uM.
	\end{equation}
	In the sequel, for any $u\in M$, we define the norm $\left\|\varphi|_{M}'(u)\right\|$ by
	\begin{equation}\label{DMN}
		\left\|\varphi|_{M}'(u)\right\|=\sup_{v\in T_uM,\|v\|_E=1}|\langle\varphi'(u), v\rangle|.
	\end{equation}
	
	\par
	Let $E\times\R$ be equipped with the scalar product
	$$
	((u,\tau),(v,\sigma))_{E\times\R}:=(u,v)_{H}+\tau\sigma, \ \  \forall \ (u,\tau), (v,\sigma)\in E\times \R,
	$$
	and the corresponding norm
	$$
	\|(u,\tau)\|_{E\times\R}:=\sqrt{\|u\|^2_{H}+\tau^2}, \ \  \forall \ (u,\tau)\in E\times \R.
	$$
	Next, we consider a functional $\tilde{\varphi}: E\times\R\rightarrow \R$ which is of class $\mathcal{C}^1$
	on $E\times\R$. We denote by $\tilde{\varphi}|_{M\times\R}$ the trace of $\tilde{\varphi}$ on $M\times\R$.
	Then $\tilde{\varphi}|_{M\times\R}$  is a $\mathcal{C}^1$ functional on $M\times\R$, and for any $(u,\tau)
	\in M\times\R$,
	\begin{equation}\label{DMM}
		\langle\tilde{\varphi}|_{M\times\R}'(u,\tau), (v,\sigma)\rangle
		:=\langle\tilde{\varphi}'(u,\tau), (v,\sigma)\rangle, \ \  \forall \ (v,\sigma)\in \tilde{T}_{(u,\tau)}(M\times\R),
	\end{equation}
	where
	\begin{equation}\label{TMM}
		\tilde{T}_{(u,\tau)}(M\times\R):=\{(v,\sigma)\in E\times\R: (u,v)_H=0\}.
	\end{equation}
	In the sequel, for any $(u,\tau)\in M\times\R$, we define the norm $\left\|\tilde{\varphi}|_{M\times\R}'(u,\tau)\right\|$ by
	\begin{equation}\label{DMN}
		\left\|\tilde{\varphi}|_{M\times\R}'(u,\tau)\right\|=\sup_{(v,\sigma)\in \tilde{T}_{(u,\tau)}(M\times\R),\|(v,\sigma)\|_{E\times\R}=1}|\langle\tilde{\varphi}'(u,\tau),
		(v,\sigma)\rangle|.
	\end{equation}

	\begin{lemma}\label{lem 2.16} {\rm\cite{BL2}} Let $\{u_n\}\subset M$ be a bounded sequence in $E$. Then
		the following are equivalent:
		\begin{enumerate}[{\rm (i)}]
			\item $\|\varphi|_{M}'(u_n)\|\rightarrow 0$ as $n\rightarrow \infty${\rm;}
			\item $\varphi'(u_n)-\langle\varphi'(u_n),u_n\rangle u_n\rightarrow 0$ in $E'$ as $n\rightarrow \infty$.
		\end{enumerate}
	\end{lemma}

	\begin{lemma}\label{lem 2.17}{\rm\cite{CTS}}
		Let $\varphi\in \mathcal{C}^1(E,\R)$ and $K\subset E$. If there exists $\rho>0$ such that
		\begin{equation}\label{D50}
			a:=\inf_{v\in M\cap K}\varphi(v)
			<b:=\inf_{v\in M\cap (K_{\rho}\setminus K)}\varphi(v),
		\end{equation}
		where  $K_{\delta}:=\{v\in E:\|u-v\|<\delta, \ \forall \ u\in K\}$, then for every $\varepsilon\in (0,(b-a)/2)$, $\delta\in (0,\rho/2)$ and $w\in M\cap A$ such that
		\begin{equation}\label{D52}
			\varphi(w)\le a+\varepsilon,
		\end{equation}
		there exists $u\in M$ such that
		\begin{enumerate}[{\rm (i)}]
			\item $a-2\varepsilon\le \varphi(u)\le a+2\varepsilon${\rm;}
			\item $\|u-w\|\le 2\delta${\rm;}
			\item $\left\|\varphi|_{M}'(u)\right\|\le 8\varepsilon/\delta$.
		\end{enumerate}
	\end{lemma}
	
	\begin{corollary}\label{cor 2.16}
		Let $\varphi\in \mathcal{C}^1(E,\R)$ and $A\subset E$. If there exist $\rho>0$ and $\bar{u}\in M\cap A$ such that
		\begin{equation}\label{D50}
			\varphi(\bar{u})=\inf_{v\in M\cap A}\varphi(v)
			<\inf_{v\in M\cap (A_{\rho}\setminus A)}\varphi(v),
		\end{equation}
		then $\varphi|_{M}'(\bar{u})=0$.
	\end{corollary}
	
	\begin{lemma}\label{lem 2.18}{\rm\cite{CTS}}
		Assume that $\tilde{\theta}\in \R$, $\tilde{\varphi}\in \mathcal{C}^1(E\times \R,\R)$ and $\tilde{\Upsilon}
		\subset M\times\R$ is a closed set. Let
		\begin{equation}\label{Gat0}
			\tilde{\Gamma}:=\left\{\tilde{\gamma}\in \mathcal{C}([0,1], M\times\R):  \tilde{\gamma}(0)\in \tilde{\Upsilon}, \ \tilde{\varphi}(\tilde{\gamma}(1)) < \tilde{\theta}\right\}.
		\end{equation}
		If $\tilde{\varphi}$ satisfies
		\begin{equation}\label{F40}
			\tilde{a}:=\inf_{\tilde{\gamma}\in \tilde{\Gamma}}\max_{t\in [0, 1]}\tilde{\varphi}(\tilde{\gamma}(t))
			>\tilde{b}:=\sup_{\tilde{\gamma}\in \tilde{\Gamma}}\max\left\{\tilde{\varphi}(\tilde{\gamma}(0)), \tilde{\varphi}(\tilde{\gamma}(1))\right\},
		\end{equation}
		then, for every $\varepsilon\in (0,(\tilde{a}-\tilde{b})/2)$, $\delta>0$ and $\tilde{\gamma}_*\in \tilde{\Gamma}$
		such that
		\begin{equation}\label{F42}
			\sup_{t\in [0, 1]}\tilde{\varphi}(\tilde{\gamma}_*(t))\le \tilde{a}+\varepsilon,
		\end{equation}
		there exists $(v,\tau)\in M\times\R$ such that
		\begin{enumerate}[{\rm (i)}]
			\item $\tilde{a}-2\varepsilon\le \tilde{\varphi}(v,\tau)\le \tilde{a}+2\varepsilon$;
			\item $\min_{t\in [0,1]}\|(v,\tau)-\tilde{\gamma}_*(t)\|_{E\times\R}\le 2\delta$;
			\item $\left\|\tilde{\varphi}|_{M\times\R}'(v,\tau)\right\|\le \f{8\varepsilon}{\delta}$.
		\end{enumerate}
	\end{lemma}
	
{\section{Proof of Theorem \ref{thm 1.1}}}
 \setcounter{equation}{0}

	\par
	 In this section, we consider the existence of a local minima for $\Phi$ in the set $\mathcal{S}_c\cap A_{s_0}$ and give the proof of Theorem \ref{thm 1.1}.
	\par
	First, we
	define a function $\zeta(t)$ as follows:
	\begin{equation}\label{xit}
		\zeta(t):=\sum_{k=0}^{\infty}\frac{4^{k+2}(k+1)+1}{(k+1)(k+3)!}t^k, \ \ \ \ t\in \R.
	\end{equation}
	Assume that {\rm(G1)} holds, then there exist two numbers $c_1\in (0,2\pi]$ and  $c_2\in (0,2\pi]$  such that
	\begin{equation}\label{c1d}
		30c^{\frac{1}{4}}\mathcal{C}_4\|g\|_{\frac{4}{3}}
  \left[\frac{2c^{\frac{5}{4}}\zeta\left(\sqrt{\f{c}{3\pi}}\right)}{3\pi^2\mathcal{C}_4\|g\|_{\frac{4}{3}}}\right]^{\f{3}{5}}\le 1, \ \ \ \ \forall\ c\in (0,c_1]
	\end{equation}
	and
	\begin{equation}\label{c2d}
		12c^{\frac{1}{4}}\left(\frac{\pi}{3}\right)^{-\frac{3}{4}}\mathcal{C}_4\|g\|_{\frac{4}{3}}
		+\f{12c^{\f{3}{2}}\left(\frac{\pi}{3}\right)^{\frac{1}{2}}}{\pi^2}\zeta\left(\sqrt{\f{c}{3\pi}}\right)\le 1, \ \ \ \ \forall\ c\in (0,c_2],
	\end{equation}
	moreover, the equality in \eqref{c1d} and \eqref{c2d}  is attained if and only if $c=c_1$ and $c=c_2$, respectively.
	Let
	\begin{equation}\label{c0}
		c_0= \min\{c_1,c_2\}\ \ \hbox{and}\ \
		s_0=\min\left\{\left[\frac{3\pi^2\mathcal{C}_4\|g\|_{\frac{4}{3}}}{2c^{\frac{5}{4}}\zeta\left(\sqrt{\f{c}{3\pi}}\right)}\right]^{\f{4}{5}},
		\frac{\pi}{3}\right\}.
	\end{equation}
	
	\par
	
	For each $c>0$, we consider the function $\tilde{h}_c(s)$ defined on $s\in (0,+\infty)$ by
	\begin{align}\label{Gct}
		\tilde{h}_c(s)  :=   \f{1}{12}-c^{\frac{1}{4}}s^{-\frac{3}{4}}\mathcal{C}_4\|g\|_{\frac{4}{3}}
		-\f{c^{\f{3}{2}}s^{\frac{1}{2}}}{\pi^2}\zeta\left(\sqrt{\f{c}{3\pi}}\right).
	\end{align}
	The following statement will help us to establish  a structure of local minima of $\Phi$ on $\mathcal{S}_c$.
 \begin{lemma}\label{lem 3.1}
  Assume that {\rm(G1)} holds. Then for each $c\in (0,c_0]$, the function  $\tilde{h}_c(s)$ satisfies
 \begin{align}\label{Gcm}
  \max_{0<s\le \pi/3}\tilde{h}_c(s)=\tilde{h}_c\left(\min\left\{s_c,\frac{\pi}{3}\right\}\right)
   \left\{
   \begin{array}{ll}
     >0, \ \mbox{if}  & c<c_0, \\
     \ge 0, \ \mbox{if}  & c=c_0,
   \end{array}
 \right.
 \end{align}
  where
 \begin{align}\label{sc}
   s_c=\left[\frac{3\pi^2\mathcal{C}_4\|g\|_{\frac{4}{3}}}{2c^{\frac{5}{4}}\zeta\left(\sqrt{\f{c}{3\pi}}\right)}\right]^{\f{4}{5}}.
 \end{align}
 \end{lemma}

 \begin{proof} By \eqref{c2d} and \eqref{Gct}, we have
 \begin{align}\label{sc2}
   \tilde{h}_c\left(\frac{\pi}{3}\right)
  & =   \f{1}{12}-c^{\frac{1}{4}}\left(\frac{\pi}{3}\right)^{-\frac{3}{4}}\mathcal{C}_4\|g\|_{\frac{4}{3}}
		-\f{c^{\f{3}{2}}\left(\frac{\pi}{3}\right)^{\frac{1}{2}}}{\pi^2}\zeta\left(\sqrt{\f{c}{3\pi}}\right)\nonumber\\
  & \left\{
   \begin{array}{ll}
     >0, \ \mbox{if}  & c<c_2, \\
     \ge 0, \ \mbox{if}  & c=c_2.
   \end{array}
 \right.
 \end{align}
 Moreover, by definition of $\tilde{h}_c(s)$, we have
 $$
    \tilde{h}_c'(s) =\frac{3c^{\frac{1}{4}}\mathcal{C}_4\|g\|_{4/3}}{4}s^{-\frac{7}{4}}
		-\f{c^{\f{3}{2}}s^{-\frac{1}{2}}}{2\pi^2}\zeta\left(\sqrt{\f{c}{3\pi}}\right).
 $$
 It is easy to check that $\tilde{h}_c'(s)=0$  has a unique solution at $s=s_c$, and by using \eqref{c1d}, we have
 \begin{align}\label{sc1}
   \max_{s> 0}\tilde{h}_c(s)
  & = \tilde{h}_c(s_c)=\f{1}{12}-\frac{5c^{\frac{1}{4}}\mathcal{C}_4\|g\|_{\frac{4}{3}}}{2}s_c^{-\f{3}{4}}\nonumber\\
  & =  \f{1}{12}-\frac{5c^{\frac{1}{4}}\mathcal{C}_4\|g\|_{\frac{4}{3}}}{2}
  \left[\frac{2c^{\frac{5}{4}}\zeta\left(\sqrt{\f{c}{3\pi}}\right)}{3\pi^2\mathcal{C}_4\|g\|_{\frac{4}{3}}}\right]^{\f{3}{5}}\nonumber\\
  & \left\{
   \begin{array}{ll}
     >0, \ \mbox{if}  & c<c_1, \\
     \ge 0, \ \mbox{if}  & c=c_1.
   \end{array}
 \right.
 \end{align}
 Since $c_0=\min\{c_1,c_2\}$, then \eqref{Gcm} follows from \eqref{sc1} and \eqref{sc2}, and hence the lemma follows.
 \end{proof}

Set
	$$
	A_{s}:=\left\{u\in H^1(\R^2): \|\nabla u\|_2^2 < s\right\}.
	$$
	
	\begin{lemma}\label{lem 3.2}
		 Assume that {\rm(G1)} holds. Then for each $c\in (0,2\pi]$, we have
		\begin{align}\label{Phs}
			\Phi(u)\ge \|\nabla u\|_2^2\ \tilde{h}_c(\|\nabla u\|_2^2),  \ \ \forall \ u\in \mathcal{S}_c\cap\overline{ A_{\pi/3}}.
		\end{align}
	\end{lemma}
	
	\begin{proof} Using Lemma \ref{lemB.1}, we can deduce that
\begin{equation}\label{Bu3}
			B(u) \le  \frac{1}{16\pi^2}\|\nabla u\|_2^2\|u\|_2^4,
			\ \ \ \ \forall\ u\in H_r^1(\R^2).
		\end{equation}
Let $u\in \mathcal{S}_c\cap \overline{A_{\pi/3}}$ and $c\in (0,2\pi]$.
		By \eqref{Bu2} and \eqref{Bu3}, we have
		\begin{align}\label{NB0}
			\|\nabla u\|_2^2+ B(u)-\frac{1}{2}\|u\|_4^4
			& \ge \|\nabla u\|_2^2+ B(u)-2\|\nabla u\|_2 B(u)^{1/2}\nonumber\\
			& \ge \left(1-\frac{c}{4\pi}\right)^2\|\nabla u\|_2^2  \ge \frac{1}{4} \|\nabla u\|_2^2.
		\end{align}
		Then  for each $c\in (0,2\pi]$, it follows from \eqref{fai}, \eqref{GNs}, \eqref{D06}, \eqref{xit}, \eqref{Gct}, \eqref{NB0} and {\rm(G1)} that
		\begin{align}\label{Phg1}
			\Phi(u)
			&  =  \frac{1}{2}\|\nabla u\|_2^2+\frac{1}{2}B(u)-\frac{1}{4}\|u\|_4^4
			-\frac{1}{2}\int_{\R^2}\left(e^{u^2}-1-u^2-\frac{u^4}{2}\right)\mathrm{d}x-\int_{\R^2}g(x)u\mathrm{d}x\nonumber\\
			& \ge  \frac{1}{8}\|\nabla u\|_2^2
			-\f{\|\nabla u\|_2^6}{4\pi(\pi-\|\nabla u\|_2^2)}-\|g\|_{\frac{4}{3}}\|u\|_4\nonumber\\
			&        \ \ \ \ -\f{c^{\f{3}{2}}\|\nabla u\|_2^3}{\pi^2}\sum_{k=0}^{\infty}\f{4^{k+2}(k+1)+1}{(k+1)(k+3)!}
			\left(\f{\|u\|_2\|\nabla u\|_2}{\pi}\right)^k\nonumber\\
			& \ge  \f{1}{12}\|\nabla u\|_2^2-c^{\frac{1}{4}}\mathcal{C}_4\|g\|_{\frac{4}{3}}\|\nabla u\|_2^{\frac{1}{2}}
			-\f{c^{\f{3}{2}}\|\nabla u\|_2^3}{\pi^2}\zeta\left(\sqrt{\f{c}{3\pi}}\right)\nonumber\\
			& =  \|\nabla u\|_2^2\left[\f{1}{12}-c^{\frac{1}{4}}\mathcal{C}_4\|g\|_{\frac{4}{3}}\|\nabla u\|_2^{-\frac{3}{2}}
			-\f{c^{\f{3}{2}}\|\nabla u\|_2}{\pi^2}\zeta\left(\sqrt{\f{c}{3\pi}}\right)\right]\nonumber\\
			&  =    \|\nabla u\|_2^2\ \tilde{h}_c(\|\nabla u\|_2^2), \ \ \forall \ u\in \mathcal{S}_c\cap \overline{A_{\pi/3}}.
		\end{align}
The the proof is completed.
	\end{proof}
	\par
Next, we consider the following local minimization problem: for each $c \in (0, c_0)$,
$$m(c):=\inf_{u\in \mathcal{S}_c\cap  A_{s_0}}\Phi(u),$$
where $c_0,s_0$ are given by \eqref{c0}.

 \begin{lemma}\label{lem 3.3}
 Assume that {\rm(G1)}  holds, then for each $c \in (0, c_0)$, there holds
 $$
    m(c)=\inf_{u\in \mathcal{S}_c\cap  A_{s_0}}\Phi(u)<0<\inf_{u\in \partial(\mathcal{S}_c\cap A_{s_0})}\Phi(u).$$
 \end{lemma}
 \begin{proof} For any $u\in \partial\left(\mathcal{S}_c\cap  A_{s_0}\right)$, we have $\|\nabla u\|_2^2=s_0$.
 Thus, using Lemmas \ref{lem 3.1} and \ref{lem 3.2}, we can deduce that
 $$
   \Phi(u)\ge \|\nabla u\|_2^2\ \tilde{h}_c(\|\nabla u\|_2^2)=s_0\tilde{h}_{c}(s_0)>s_0\tilde{h}_{c_0}(s_0) \ge 0.
 $$
 For any $u\in \mathcal{S}_c$, we define the function $H_{u} : (0, +\infty ) \rightarrow \R$ by
		\begin{align}\label{hud11}
			H_{u}(t)
			&:  =  \Phi(tu_t)\nonumber\\
&\ =\f{t^2}{2}\left[\|\nabla u\|_2^2+B(u)\right]
			-\f{1}{2t^{2}}\int_{\R^2}\left(e^{t^2u^2}-1-t^2u^2\right)\mathrm{d}x-\f{1}{t}\int_{\R^2}g\left(\frac{x}{t}\right)u\mathrm{d}x.
		\end{align}
  By {\rm(G1)}, without loss of generality, we can assume that $ g(x_0) > 0$ with $x_0 \in \mathbb{R}^2 $. Using $g \in \mathcal {C}(\mathbb{R}^2,[0,+\infty))$, it is easy to derive that
 there exists $0 < \delta \le 1$ such that
 \begin{align}\label{hud2}
 g(x) \ge \frac{g(x_0)}{2}>0,  \ \ \ \ \forall \ x \in B_\delta(x_0).
  \end{align}
 Take $\tilde{u} \in  H^1(\mathbb{R}^2) \cap \mathcal {C}(\mathbb{R}^2,\mathbb{R}), \tilde{u}>0 $ and $\|\tilde{u}\|^2_2=c<c_0,$ then $\tilde{u}$ is bounded from positive below if $|x| \le |x_0|+1 $. So we define
 \begin{align}\label{hud3}
 \tilde{m}:=\min\{ \tilde{u}(x): |x| \le |x_0|+1\}>0.
  \end{align}
When $t \in (0,1]$, from \eqref{hud11}, \eqref{hud2} and \eqref{hud3}, we have
 \begin{align}\label{hud1}
 H_{\tilde{u}}(t)&=\f{t^2}{2}\left[\|\nabla \tilde{u}\|_2^2+B(\tilde{u})\right]
     -\f{1}{2t^{2}}\int_{\R^2}\left(e^{t^2\tilde{u}^2}-1-t^2\tilde{u}^2\right)\mathrm{d}x-\f{1}{t}\int_{\R^2}g\left(\frac{x}{t}\right)\tilde{u}\mathrm{d}x \nonumber\\
     & \le \f{t^2}{2}\left[\|\nabla \tilde{u}\|_2^2+B(\tilde{u})\right]-\f{1}{t}\int_{B_{t\delta}(tx_0)}\frac{g(x_0)}{2}\tilde{u}\mathrm{d}x \nonumber\\
      & \le \f{t^2}{2}\left[\|\nabla \tilde{u}\|_2^2+B(\tilde{u})\right]-\f{1}{t}\int_{B_{t\delta}(tx_0)}\frac{g(x_0)}{2}\tilde{m}\mathrm{d}x \nonumber\\
      & \le \f{t^2}{2}\left[\|\nabla \tilde{u}\|_2^2+B(\tilde{u})\right]-\frac{1 }{2}t \pi \delta^2 g(x_0)\tilde{m} := \bar{H}_{\tilde{u}}(t) .
 \end{align}
In the above three inequality, we use the following fact:
$$B_{t\delta}(tx_0) \subset  B_{|x_0|+\delta}(0)\subset  B_{|x_0|+1}(0), \ \  \hbox{for} \ \  t,\delta \in (0,1] .$$
It follows from \eqref{hud1} that $\lim_{t\to 0^{+}}\bar{H}_{\tilde{u}}(t) = 0^{-}$. Therefore, there exists $t_0 > 0$ small
 enough such that $\|\nabla (t_0 \tilde{u}_{t_0})\|_2^2=t_0^2\|\nabla \tilde{u}\|_2^2<s_0$ and $\Phi(t_0 \tilde{u}_{t_0})=H_{\tilde{u}}(t_0)\leq \bar{H}_{\tilde{u}}(t)<0$. This,
 together with $t_0\tilde{u}_{t_0}\in \mathcal{S}_c$, implies that $m(c)<0$.
 \end{proof}
\begin{remark}\label{rem 1.1}
There are many functions which satisfy $u \in  H^1(\mathbb{R}^2) \cap \mathcal {C}(\mathbb{R}^2,\mathbb{R}), u>0 $ and $\|u\|^2_2<c_0$. For example,
\begin{equation}\label{better}
  u(x)=\frac{\sqrt{c_0/3\pi}}{|x|^4+1}.
\end{equation}

\end{remark}
 Applying Lemmas \ref{lem 2.17} and \ref{lem 3.3}, we obtain the existence of a  minimizing sequence for $ m(c)$, which is also
a (PS) sequence for $\Phi$ on $\mathcal{S}_c$, which reads as follows.

 \begin{lemma}\label{lem 3.4} Assume that {\rm(G1)} and {\rm(G2)} hold, then  for any $c\in (0,c_0)$, there exists a sequence $\{u_n\}\subset \mathcal{S}_c$ such that
 \begin{equation}\label{mcn}
    \|\nabla u_n\|_2^2< s_0\leq\frac{\pi}{3}, \ \ \Phi(u_n)\rightarrow m(c)<0 \   \
   \mbox{and}\ \ \Phi|_{\mathcal{S}_c}'(u_n) \rightarrow 0.
 \end{equation}
 \end{lemma}

	\begin{proof}[Proof of Theorem {\rm\ref{thm 1.1}}] Let $\{u_n\}\subset \mathcal{S}_c$ be a minimizing sequence of $m(c)$ given by \eqref{mcn}. Without loss of generality, we can assume that $u_n \geq 0$.
	By Lemma \ref{lem 3.4}, it is easy to get that $\{u_n\}$ is bounded in $H_r^1(\R^2)$. Then there exists $\bar{u}\in H_{r}^1(\R^2)$ such that, passing to a subsequence,
		\begin{align}\label{BU1}
			u_n\rightharpoonup \bar{u}\ \ \mbox{in}\ H_{r}^1(\R^2), \ \ u_n\rightarrow \bar{u}
			\ \ \mbox{in}\ L^s(\R^2) \ \mbox{for} \ s>2, \ \
			u_n\rightarrow \bar{u}\ \mbox{a.e.\ on}\ \R^2.
		\end{align}
To derive the existence of normalized solution which  is a minimizer of $\Phi$ in the set $\mathcal{S}_c\cap A_{s_0}$,  we divide into the following two assertions.
		\vskip2mm
 {\bf Assertion 1.} $\bar{u}\ne 0 $.

 \vskip2mm
Otherwise, we suppose that  $\bar{u}=0$. By \eqref{BU1}, we have
		\begin{equation}\label{G30}
			\int_{\R^2}|u_n|^s\mathrm{d}x=o(1).
		\end{equation}
		Besides, from the H\"older inequality, {\rm(G1)}, Lemma \ref{lem 2.3} and the fact that $\|\nabla u_n\|_2^2<\pi/3$, we derive that
		\begin{align}\label{glim}
			\int_{\R^2}g(x)u_n\mathrm{d}x  \le  \|g(x)\|_{\frac{4}{3}}\|u_n\|_4 = o(1)
		\end{align}
and
\begin{align}\label{G32}
			\int_{\R^2}\left(e^{u_n^2}-1-u_n^2\right)\mathrm{d}x
			&  \le  \int_{\R^2}u_n^2\left(e^{u_n^2}-1\right)\mathrm{d}x\nonumber\\
			& \le  \left[\int_{\R^2}\left(e^{2u_n^2}-1\right)\mathrm{d}x\right]^{\frac{1}{2}}\|u_n\|_4^2=o(1).
		\end{align}
		By \eqref{fai}, \eqref{mcn}, \eqref{glim} and \eqref{G32}, we have
		\begin{align*}\label{G34}
			0 & \le \|\nabla u_n\|_2^2+B(u_n)\nonumber\\
			&  = 2m(c)
			+\int_{\R^2}\left(e^{u_n^2}-1-u_n^2\right)\mathrm{d}x+2\int_{\R^2}g(x)u_n\mathrm{d}x+o(1)\nonumber\\
			&  = 2m(c)+o(1).
		\end{align*}
		This is contradictory  due to $m(c)<0$, and so $\bar{u}\ne0$.
			\vskip2mm
 {\bf Assertion 2.} $u_n\rightarrow \bar{u}$ in $H_{r}^1(\R^2)$.

 \vskip2mm
		By \eqref{mcn} and Lemma \ref{lem 2.16}, we have
		\begin{equation}\label{G43}
			\Phi'(u_n)+\lambda_n{u}_n\rightarrow 0,
		\end{equation}
		where
		\begin{align}\label{G430}
			-\lambda_n&=\frac{1}{\|{u}_n\|_2^2}\langle\Phi'({u}_n),{u}_n\rangle\nonumber\\
			&=\frac{1}{c}\left[\|\nabla {u}_n\|_2^2+3B({u}_n)
			-\int_{\R^2}\left(e^{{u}_n^2}-1\right){u}_n^2\mathrm{d}x-\int_{\R^2}g(x)u_n\mathrm{d}x\right].
		\end{align}
		Since $\{\|{u}_n\|\}$ is bounded, we derive that $\{|\lambda_n|\}$ is bounded.
		We may thus assume, passing to a subsequence if necessary, that $\lambda_n\rightarrow \bar{\lambda}_c$.
		From \eqref{G43}, Lemmas \ref{lem2.1} and \ref{lem 2.3}, it is easy to check that
		\begin{equation}\label{G44}
			\Phi'(\bar{u})+\bar{\la}_c\bar{u}=0.
		\end{equation}
		By \eqref{G44} and Lemma \ref{Lep}, we have
		\begin{equation}\label{G45}
			\mathcal{P}(\bar{u}) = \|\nabla \bar{u}\|_2^2+ B(\bar{u})
			-\int_{\R^2}\left[\left(\bar{u}^2-1\right)e^{\bar{u}^2}+1\right]\mathrm{d}x+\int_{\R^2}\left[g(x)+\nabla g(x)\cdot x\right]\bar{u}\mathrm{d}x=0.
		\end{equation}
		From \eqref{Bu1},  \eqref{G44}, \eqref{G45} and {\rm(G2)}, we derive that
		\begin{align}\label{lamd}
			0
			&  =    \left\langle \Phi'(\bar{u}),\bar{u}\right\rangle+\bar{\la}_c\|\bar{u}\|_2^2-\mathcal{P}(\bar{u})\nonumber\\
			&  = \bar{\la}_cc+2B(\bar{u})-\int_{\R^2}\left(e^{\bar{u}^2}-1-\bar{u}^2\right)\mathrm{d}x-\int_{\R^2}\left[2g(x)+\nabla g(x)\cdot x\right]\bar{u}\mathrm{d}x\nonumber\\
			& \le \bar{\la}_cc+\frac{1}{8\pi}c\|\bar{u}\|_4^4-\frac{1}{2}\|\bar{u}\|_4^4\nonumber\\
			& \le \bar{\la}_cc-\frac{1}{4}\|\bar{u}\|_4^4  \ \ \mbox{for}\ \ c\in (0,2\pi],
		\end{align}
		which, together with $\bar{u}\ne0 $ and $0<c<c_0\le 2\pi$, yields
		\begin{equation}\label{G46}
			\bar{\la}_cc
			\ge  \frac{1}{4}\|\bar{u}\|_4^4  >0  \ \ \mbox{for}\ \ c\in (0,c_0).
		\end{equation}
		It follows from \eqref{mcn}, Lemmas \ref{lem2.1} and \ref{lem 2.3} that
		\begin{equation}\label{G47}
			\left\langle B'({u}_n),{u}_n-\bar{u}\right\rangle=o(1)
		\end{equation}
		and
		\begin{equation}\label{G48}
			\int_{\R^2}\left(e^{{u}_n^2}-1\right){u}_n({u}_n-\bar{u})\mathrm{d}x
			= o(1).
		\end{equation}
		Combining \eqref{G43}, \eqref{G44}, \eqref{G46},  \eqref{G47} and \eqref{G48}, we have
		\begin{align}\label{G49}
			o(1)
			& =   \left\langle  \Phi'({u}_n)+\la_n {u}_n,{u}_n-\bar{u}\right\rangle  \nonumber\\
			&   =   \|\nabla({u}_n-\bar{u})\|_2^2+\bar{\la}_c\|{u}_n-\bar{u}\|_2^2+o(1).
		\end{align}
		Using  \eqref{G46}, it is easy to check that $u_n\to \bar{u}$ in $H^1_r(\R^2)$. The proof of Assertion 2 is completed and thus this completes the proof of Theorem \ref{thm 1.1}.
	\end{proof}
{\section{Proof of Theorem \ref{thm 1.2}}}
 \setcounter{equation}{0}
	
	This section is devoted to consider that the existence of a critical point of mountain-pass type for $\Phi$ on $\mathcal{S}_c$, and we will give the proof of Theorem \ref{thm 1.2}.
	
By  \cite[Proposition 2.1]{JL-MA} and Theorem \ref{thm 1.1}, we can easily prove the following Lemma:
	
	\begin{lemma}\label{lem 3.6} Assume that {\rm(G1)-(G2)} hold and $c \in (0,c_0)$, then $m(c)$ is reached by a positive, radially symmetric non-increasing function, denoted $u_c$ that satisfies, for a  constant $\lambda_c> 0$,
		\begin{equation}\label{Pa2}
			-\Delta u_c+\left(\frac{h^{2}(\vert x\vert)}{\vert x\vert^{2}}+\int_{\vert x\vert}^{\infty}\frac{h(s)}{s}u_c^{2}(s)\mathrm{d}s\right)u_c
		-\left(e^{u_c^2}-1\right)u_c-g(x)=-\lambda_cu_c \ \   \mbox{in}\ \R^2.
		\end{equation}
	\end{lemma}

	\begin{lemma}\label{lem 3.8}
		Assume that {\rm(G1)-(G2)} hold and $c \in (0, c_0)$. Then there exists $\kappa_c>0$ such that
		\begin{equation}\label{Mu2}
			M(c):=\inf_{\gamma\in \Gamma_{c}}\max_{t\in [0, 1]}\Phi(\gamma(t))\ge \kappa_c
			>\sup_{\gamma\in \Gamma_{c}}\max\left\{\Phi(\gamma(0)), \Phi(\gamma(1))\right\},
		\end{equation}
		where
		\begin{equation}\label{Ga1}
			\Gamma_{c}=\left\{\gamma\in \mathcal{C}([0,1], \mathcal{S}_c): \gamma(0)=u_c, \Phi(\gamma(1))<2m(c)\right\}.
		\end{equation}
	\end{lemma}
	
	\begin{proof}
		Set $\kappa_c:=\inf_{u\in \partial (\mathcal{S}_c\cap A_{s_0})}\Phi(u)$. By Lemma \ref{lem 3.3}, we have
		$\kappa_c>0$. Let $\gamma\in \Gamma_{c}$ be arbitrary. Since $\gamma(0)=u_c\in \mathcal{S}_c\cap A_{s_0}$ and
		$\Phi(\gamma(1))<2m(c)<m(c)$, then  Lemma \ref{lem 3.3} implies that $\gamma(1)
		\not\in \mathcal{S}_c\cap A_{s_0}$. By continuity of $\gamma(t)$ on $[0,1]$, there exists
		$t_0 \in (0,1)$ such that $\gamma(t_0)\in \partial (\mathcal{S}_c\cap A_{s_0})$, and so
		$\max_{t\in [0, 1]}\Phi(\gamma(t))\ge \kappa_c$. Hence, \eqref{Mu2} holds and the proof is completed.
	\end{proof}
	
	\par
	To apply Lemma \ref{lem 2.18}, we let $E=H_{r}^1(\R^2)$ and $H=L^2(\R^2)$. Define
	the norms of $E$ and $H$ by
	\begin{equation}\label{E12}
		\|u\|_E:= \left[\int_{\R^2}\left(|\nabla u|^{2}+u^2\right)\mathrm{d}x\right]^{1/2}, \ \ \ \ \|u\|_H:=\frac{1}{\sqrt{c}}\left(\int_{\R^2}u^2\mathrm{d}x\right)^{1/2}, \ \ \forall \ u\in E.
	\end{equation}
	After identifying $H$ with its dual, we have $E\hookrightarrow H \hookrightarrow E^*$ with continuous injections.
	Set
	\begin{equation}\label{E14}
		M:= \left\{u\in E: \|u\|_2^2=\int_{\R^2}u^2\mathrm{d}x=c\right\}.
	\end{equation}
	Set $F(u):=\frac{1}{2}\left(e^{u^2}-1-u^2\right)$ and $f(u):=\left(e^{u^2}-1\right)u$.
	Let us define a continuous map $ \beta: H_{r}^1(\R^2)\times\R \to H_{r}^1(\R^2)$ by
	\begin{equation}\label{Be}
		\beta(v, t)(x):=e^{t}v(e^tx)\ \ \mbox{for}\  v\in H_{r}^1(\R^2),\ \  \forall\ t\in\R,\ x\in\R^2,
	\end{equation}
	and consider the following auxiliary functional:
	\begin{equation}\label{Tbe}
		\tilde{\Phi}(v,t):=\Phi(\beta(v,t))
		=\f{e^{2t}}{2}\left(\|\nabla v\|_2^2+B(v)\right)
		-\f{1}{e^{2t}}\int_{\R^2}F(e^{t}v)\mathrm{d}x-e^{t}\int_{\R^2}g(x)v(e^tx)\mathrm{d}x.
	\end{equation}
	Using \eqref{Pu}, \eqref{Phd}, \eqref{Phd1}, \eqref{Tbe}, {\rm(G1)} and {\rm(G2)}, it is easy to check that $\tilde{\Phi}'$ is of class $\mathcal{C}^1$, and for any $(w,s)\in H_{r}^1(\R^2)\times \R$,
	\begin{equation}\label{DIv}
		\left\langle\tilde{\Phi}'(v,t),(w,s)\right\rangle
		=  \left\langle\Phi'(\beta(v,t)),\beta(w,t)\right\rangle+s\mathcal{P}(\beta(v,t)).
	\end{equation}
	Let
	\begin{equation}\label{uph}
		u(x):=\beta(v,t)(x)=e^{t}v(e^tx), \ \  \phi(x):=\beta(w,t)(x)=e^{t}w(e^tx).
	\end{equation}
	Then
	\begin{equation}\label{E20}
		(u,\phi)_H=\f{1}{c}\int_{\R^2}u(x)\phi(x)\mathrm{d}x=\f{1}{c}\int_{\R^2}v(x)w(x)\mathrm{d}x
		=(v,w)_H.
	\end{equation}
	This shows that
	\begin{equation}\label{E21}
		\phi\in T_u(\mathcal{S}_c) \ \Leftrightarrow \ (w,s)\in \tilde{T}_{(v,t)}(\mathcal{S}_c\times\R),
		\ \  \forall \ t,s\in \R.
	\end{equation}
	It follows from \eqref{DIv}, \eqref{uph} and \eqref{E21} that
	\begin{equation}\label{E24}
		|\mathcal{P}(u)|=\left|\left\langle\tilde{\Phi}'(v,t),(0,1)\right\rangle\right|
		\le \left\|\tilde{\Phi}|_{\mathcal{S}_c\times\R}'(v,t)\right\|
	\end{equation}
	and
	\begin{eqnarray}\label{E25}
		\left\|\Phi|_{\mathcal{S}_c}'(u)\right\|
		&  =  & \sup_{\phi\in T_u(\mathcal{S}_c)}\f{1}{\|\phi\|_E}
		\left|\left\langle\Phi'(u),\phi\right\rangle\right|\nonumber\\
		&  =  & \sup_{\phi\in T_u(\mathcal{S}_c)}\f{1}{\sqrt{\|\nabla \phi\|_2^2+\|\phi\|_2^2}}
		\left|\left\langle\Phi'(\beta(v,t)),\beta(w,t)\right\rangle\right|\nonumber\\
		&  =  & \sup_{\phi\in T_u(\mathcal{S}_c)}\f{1}{\sqrt{\|\nabla \phi\|_2^2+\|\phi\|_2^2}}
		\left|\left\langle\tilde{\Phi}'(v,t),(w,0)\right\rangle\right|\nonumber\\
		& \le & \sup_{(w,0)\in \tilde{T}_{(v,t)}(\mathcal{S}_c\times\R)}\f{e^{|t|}}{\|(w,0)\|_{E\times\R}}
		\left|\left\langle\tilde{\Phi}'(v,t),(w,0)\right\rangle\right|\nonumber\\
		& \le & e^{|t|}\left\|\tilde{\Phi}|_{\mathcal{S}_c\times\R}'(v,t)\right\|.
	\end{eqnarray}
\par
Similar to \cite[Theorem 1.2]{ZWM} and \cite[Proposition 2.1]{CTS}, we can prove the following lemma.
	\begin{lemma}\label{lem 3.9}
		Assume that {\rm(G1)-(G2)} hold and $c \in (0, c_0)$, then there exists a nonnegative sequence $\{u_n\}\subset \mathcal{S}_c$ such that
		\begin{equation}\label{PCe2}
			\Phi(u_n)\rightarrow M(c)>0, \ \   \Phi|_{\mathcal{S}_c}'(u_n) \rightarrow 0\ \
			\mbox{and}\ \ \mathcal{P}(u_n)\rightarrow 0.
		\end{equation}
	\end{lemma}

\par
	To give the sharp energy estimate of $M(c)$,  we define the following Moser type functions $\tilde{w}_n(x)$ supported in $B_1(0)$:
	\begin{equation}\label{wnt}
		\tilde{w}_n(x)=\frac{1}{\sqrt{2\pi}}
		\begin{cases}
			\sqrt{\log n}, \ \ & 0\le |x|\le 1/n;\\
			\frac{\log(1/|x|)}{\sqrt{\log n}}, \ \ & 1/n\le |x|\le 1;\\
			0, \ \ & |x|\ge 1.
		\end{cases}
	\end{equation}
	Computing directly, we get that
	\begin{equation}\label{wnt1}
		\|\nabla \tilde{w}_n\|_2^2=\int_{\R^2}|\nabla \tilde{w}_n|^2\mathrm{d}x = 1,
	\end{equation}
	\begin{align}\label{wnt2}
		\|\tilde{w}_n\|_2^2
		&  = \int_{\R^2}|\tilde{w}_n|^2\mathrm{d}x
		=  \log n\int_{0}^{1/n}r\mathrm{d}r+\int_{1/n}^{1}\frac{\log^2(1/r)}{\log n}r\mathrm{d}r\nonumber\\
		&  =  \frac{1}{4\log n}-\frac{1}{4n^2\log n}-\frac{1}{2n^2}
	\end{align}
	and
	\begin{align}\label{wnt3}
		\|\tilde{w}_n\|_4^4
&= \int_{\R^2}\tilde{w}_n^4\mathrm{d}x=  \frac{\log^2 n}{2\pi}\int_{0}^{1/n}r\mathrm{d}r+\frac{1 }{2\pi\log^2 n }\int_{1/n}^{1}\log^4(1/r)r\mathrm{d}r\nonumber\\
		&=\frac{1}{2\pi}\left[\frac{3}{4(\log n)^2}-\frac{\log n}{n^2}-\frac{3}{2n^2}-\frac{3}{2n^2\log n}-\frac{3}{4n^2(\log n)^2}\right].
	\end{align}
	Moreover, \eqref{Bu1}, \eqref{wnt2} and \eqref{wnt3} give
	\begin{equation}\label{wnt4}
		B(\tilde{w}_n)\le \frac{1}{16\pi}\|\tilde{w}_n\|_2^2\|\tilde{w}_n\|_4^4
		\le O\left(\frac{1}{\log^3 n}\right).
	\end{equation}
	
	\begin{lemma}\label{lem 3.14}
		Assume that {\rm(G1)-(G4)} hold. Then for each $c \in (0, c_0)$, there holds
		\begin{equation}\label{Mmc}
			M(c)< m(c)+2\pi.
		\end{equation}
	\end{lemma}
	
	\begin{proof} By Lemmas \ref{Lep} and \ref{lem 3.6}, we have
		\begin{equation}\label{M05}
			u_c\in H_{r}^1(\R^2), \ \ \|u_{c}\|_2^2=c, \ \ \ \ \Phi(u_{c})=m(c),
			\ \ \ \ u_{c}(x)>0, \ \ \forall \ x\in \R^2,
		\end{equation}
		\begin{align}\label{M28}
			&  \int_{\mathbb{R}^{2}}\frac{u_{c}^{2}}{|x|^{2}}\left(\int_{0}^{|x|}\frac{s}{2}u_{c}^{2}(s)\mathrm{d}s\right)
			\left(\int_{0}^{|x|}su_{c}(s)\tilde{w}_n(s)\mathrm{d}s\right)\mathrm{d}x
			+\int_{\mathbb{R}^{2}}\frac{u_c\tilde{w}_n}{|x|^{2}}\left(\int_{0}^{|x|}\frac{s}{2}u_{c}^{2}(s)\mathrm{d}s\right)^2\mathrm{d}x\nonumber\\
			&  =   \int_{\R^2}\left[\left(e^{u_c^2}-1\right)u_c-\lambda_cu_c\right]\tilde{w}_n\mathrm{d}x+\int_{\R^2}g(x)\tilde{w}_n\mathrm{d}x-\int_{\R^2}\nabla u_{c}\cdot \nabla \tilde{w}_n\mathrm{d}x
		\end{align}
		and
		\begin{equation}\label{M06}
			\lambda_cc=\int_{\R^2}\left(e^{u_{c}^2}-1-u_{c}^2\right)\mathrm{d}x+\int_{\R^2}\left[2g(x)+\nabla g(x)\cdot x\right]u_c\mathrm{d}x-2B(u_c).
		\end{equation}
		Moreover, using {\rm(G1)}, \eqref{wnt2} and Lemma \ref{lem 2.3} i), we can easily check that
		\begin{equation}\label{M25+}
			\int_{\R^2}u_{c}\tilde{w}_n\mathrm{d}x
			\le   \|u_c\|_{2}\|\tilde{w}_n\|_2=O\left(\frac{1}{\sqrt{\log n}}\right),  \ \ n\to\infty,
		\end{equation}
\begin{equation}\label{M251}
			\int_{\R^2}g(x)\tilde{w}_n\mathrm{d}x
			\le   \|g(x)\|_{\frac{4}{3}}\|\tilde{w}_n\|_4=O\left(\frac{1}{\sqrt{\log n}}\right),  \ \ n\to\infty
		\end{equation}
		and
		\begin{equation}\label{M26}
			\int_{\R^2}\left(e^{u_c^2}-1\right)u_c\tilde{w}_n\mathrm{d}x
			=   O\left(\frac{1}{\sqrt{\log n}}\right),  \ \ n\to\infty.
		\end{equation}
		It follows from  \eqref{wnt2} and \eqref{M05}  that
		\begin{align}\label{M29}
			\|u_{c}+t\tilde{w}_n\|_2^2
			&  =    c+t^2\|\tilde{w}_n\|_2^2+2t\int_{\R^2}u_c\tilde{w}_n\mathrm{d}x\nonumber\\
			&  =    c+2t\int_{\R^2}u_c\tilde{w}_n\mathrm{d}x +t^2\left[O\left(\frac{1}{\log n}\right)\right], \ \ n\to\infty.
		\end{align}
		Let $\tau:=\|u_{c}+t\tilde{w}_n\|_2/\sqrt{c}$. Then, combining  \eqref{M25+} and \eqref{M29}, we obtain that  for any $p>0$,
		\begin{equation}\label{tau122}
			\tau^p = 1+\frac{pt}{c}\int_{\R^2}u_c\tilde{w}_n\mathrm{d}x +t^2\left[O\left(\frac{1}{\log n}\right)\right]
			, \ \ \ \ \mbox {for large } n\in \N.
		\end{equation}
In particular,  for the case $p=2$, we have
\begin{equation}\label{tau1}
			\tau^2 = 1+\frac{2t}{c}\int_{\R^2}u_c\tilde{w}_n\mathrm{d}x +t^2\left[O\left(\frac{1}{\log n}\right)\right]
			\le 1+t+t^2, \ \ \ \ \mbox {for large } n\in \N.
		\end{equation}
Now, we define
		\begin{equation}\label{Wn1}
			W_{n,t}(x):=u_{c}(\tau x)+t\tilde{w}_n(\tau x).
		\end{equation}
		Then
		\begin{equation}\label{Wnc}
			\|W_{n,t}\|_2^2=\tau^{-2}\|u_{c}+t\tilde{w}_n\|_2^2=c.
		\end{equation}
This means $W_{n,t}\in \mathcal{S}_c$ for all $t>0$.
		Besides, it is  easy to check that
		\begin{equation}\label{Wn2}
			\|\nabla W_{n,t}\|_2^2=\|\nabla(u_{c}+t\tilde{w}_n)\|_2^2
		\end{equation}
		and
		\begin{equation}\label{Wn4}
			\int_{\R^2}\left(e^{W_{n,t}^2}-1-W_{n,t}^2\right)\mathrm{d}x
			=   \f{1}{\tau^{2}}\int_{\R^2}\left[e^{(u_{c}+t\tilde{w}_n)^2}-1-(u_{c}+t\tilde{w}_n)^2\right]\mathrm{d}x.
		\end{equation}
		Next, we shall construct a fine path belonging to $\Gamma_c$ with the help of $W_{n,t}$ to establish \eqref{Mmc}.
		To this end, we first  give an upper estimate of $\max_{t\ge 0}\Phi(W_{n,t})$.
		From \eqref{fai}, \eqref{Wn1}-\eqref{Wn4} and {\rm(G1)},
		we derive that for all $t> 0$,
\begin{align}\label{H3011}
&       \Phi(W_{n,t})\nonumber\\
			&  =    \frac{1}{2}\|\nabla W_{n,t}\|_2^{2}+\frac{1}{2}B(W_{n,t})-\int_{\R^2} g(x)W_{n,t}\mathrm{d}x
			-\frac{1}{2}\int_{\R^2}\left(e^{W_{n,t}^2}-1-W_{n,t}^2\right)\mathrm{d}x\nonumber\\
			&  =    \frac{1}{2}\|\nabla (u_{c}+t\tilde{w}_n)\|_2^{2}+\frac{1}{2}\tau^{-4}B(u_{c}+t\tilde{w}_n)
			-\int_{\R^2} g(x) \left[  u_{c}(\tau x)+t\tilde{w}_n(\tau x) \right]   \mathrm{d}x\nonumber\\
			&    \ \ \ -\frac{1}{2}\tau^{-2}\int_{\R^2}\left[e^{(u_{c}+t\tilde{w}_n)^2}-1-(u_{c}+t\tilde{w}_n)^2\right]\mathrm{d}x.
\end{align}
Moreover, it follows  from \eqref{Bu}  and \eqref{tau122} that
		\begin{align}\label{BWt}
			& B(u_{c}+t\tilde{w}_n)-B(u_{c})\nonumber\\
			& = \int_{\R^2}\left[\frac{(u_{c}+t\tilde{w}_n)^{2}}{|x|^{2}}\left(\int_0^{|x|}\frac{s}{2}(u_{c}(s)+t\tilde{w}_n(s))^{2}\mathrm{d}s\right)^{2}
			-\frac{u_{c}^{2}}{|x|^{2}}\left(\int_0^{|x|}\frac{s}{2}u_{c}^{2}(s)\mathrm{d}s\right)^{2}\right]\mathrm{d}x\nonumber\\
			& = \int_{\R^2}\left\{\frac{(u_{c}+t\tilde{w}_n)^{2}-u_{c}^2}{|x|^{2}}\left(\int_0^{|x|}\frac{s}{2}(u_{c}(s)+t\tilde{w}_n(s))^{2}\mathrm{d}s\right)^{2}\right.\nonumber\\
			& \ \ \ \  +\left.\frac{u_{c}^{2}}{|x|^{2}}\left[\left(\int_0^{|x|}\frac{s}{2}(u_{c}(s)+t\tilde{w}_n(s))^{2}\mathrm{d}s\right)^{2}
			-\left(\int_0^{|x|}\frac{s}{2}u_{c}^{2}(s)\mathrm{d}s\right)^{2}\right]\right\}\mathrm{d}x\nonumber\\
			& = \int_{\R^2}\left\{\frac{2tu_{c}\tilde{w}_n+t^2\tilde{w}_n^2 }{|x|^{2}}\left(\int_0^{|x|}\frac{s}{2}(u_{c}(s)+t\tilde{w}_n(s))^{2}\mathrm{d}s\right)^{2}\right.\nonumber\\
			& \ \ \ \  +\left.\frac{u_{c}^{2}}{|x|^{2}}\left[\left(\int_0^{|x|}\frac{s}{2}(u_{c}(s)+t\tilde{w}_n(s))^{2}\mathrm{d}s\right)^{2}
			-\left(\int_0^{|x|}\frac{s}{2}u_{c}^{2}(s)\mathrm{d}s\right)^{2}\right]\right\}\mathrm{d}x\nonumber\\
			& = \int_{\R^2}\left\{\frac{2tu_{c}\tilde{w}_n+t^2\tilde{w}_n^2 }{|x|^{2}}
			\left(\int_0^{|x|}\frac{s}{2}\left(u_{c}^2(s)+2tu_{c}(s)\tilde{w}_n(s)+t^2\tilde{w}^{2}_n(s)\right)\mathrm{d}s\right)^{2}\right.\nonumber\\
			& \ \ \ \  +\left.\frac{u_{c}^{2}}{|x|^{2}}\int_0^{|x|}\frac{s}{2}\left[2tu_{c}(s)\tilde{w}_n(s)+t^2\tilde{w}^{2}_n(s)\right]\mathrm{d}s
			\int_0^{|x|}\frac{s}{2}\left[2u_{c}^2(s)+2tu_{c}(s)\tilde{w}_n(s)+t^2\tilde{w}^{2}_n(s)\right]\mathrm{d}s\right\}\mathrm{d}x\nonumber\\
			& =  2t\int_{\mathbb{R}^{2}}\frac{u_c\tilde{w}_n}{|x|^{2}}\left(\int_{0}^{|x|}\frac{s}{2}u_{c}^{2}(s)\mathrm{d}s\right)^2\mathrm{d}x
			+2t\int_{\mathbb{R}^{2}}\frac{u_{c}^{2}}{|x|^{2}}\left(\int_{0}^{|x|}\frac{s}{2}u_{c}^{2}(s)\mathrm{d}s\right)
			\left(\int_{0}^{|x|}su_{c}(s)\tilde{w}_n(s)\mathrm{d}s\right)\mathrm{d}x\nonumber\\
			& \ \ \ \  +t^2\left[O\left(\f{1}{\log n}\right)\right]
			+t^3\left[O\left(\f{1}{\log^{3/2} n}\right)\right]+t^4\left[O\left(\f{1}{\log^{2} n}\right)\right],
		\end{align}
		where, to derive the last equality,  we used the H\"older  inequality and
		\begin{align*}
			\int_{0}^{|x|}su(s)v(s)\mathrm{d}s
			& \le  \left(\int_{0}^{|x|}s \mathrm{d}s\right)^{1/2} \left(\int_{0}^{|x|}su^2(s)v^2(s)\mathrm{d}s\right)^{1/2}\\
			& \leq \frac{1}{2\sqrt{\pi}}|x|\left(\int_{\R^2}u^2v^2\mathrm{d}x\right)^{1/2}
			\le \frac{1}{2\sqrt{\pi}}|x|\|u\|_4\|v\|_4, \ \ \ \ \forall\ u,v \in H^1_r(\R^2).
		\end{align*}
Let us define the following function:
		\begin{equation}\label{Pst}
			\Psi_n(t) :=   \frac{t^2}{4}
			-\frac{1}{2\tau^2}\int_{\R^2}\left(e^{t^2\tilde{w}_n^2}-1-t^2\tilde{w}_n^2\right)\mathrm{d}x, \ \ \forall \ t >0.
		\end{equation}
Thus, from  \eqref{M05}-\eqref{M06} and \eqref{tau122}-\eqref{Pst}, we obtain that for any $t>0$,
		\begin{align}\label{H30}
			&       \Phi(W_{n,t})\nonumber\\
	& \le   \frac{1}{2}\|\nabla u_{c}\|_2^{2}+\frac{1}{2}\tau^{-4}B(u_{c})
			-\frac{1}{2\tau^{2}}\int_{\R^2}\left(e^{u_{c}^2}-1-u_{c}^2\right)\mathrm{d}x-\int_{\R^2} g(x)\left[u_{c}(\tau x)+t\tilde{w}_n(\tau x)\right]\mathrm{d}x\nonumber\\
			&       \ \ \ \  +\frac{t^2}{2}\|\nabla \tilde{w}_n\|_2^{2}
			-\frac{1}{2}\tau^{-2}\int_{\R^2}\left(e^{t^2\tilde{w}_n^2}-1-t^2\tilde{w}_n^2\right)\mathrm{d}x
            +t\tau^{-4}\int_{\mathbb{R}^{2}}\frac{u_c\tilde{w}_n}{|x|^{2}}\left(\int_{0}^{|x|}\frac{s}{2}u_{c}^{2}(s)\mathrm{d}s\right)^2\mathrm{d}x\nonumber\\
			&      \ \ \ \  +t\tau^{-4} \int_{\mathbb{R}^{2}}\frac{u_{c}^{2}}{|x|^{2}}\left(\int_{0}^{|x|}\frac{s}{2}u_{c}^{2}(s)\mathrm{d}s\right)
			\left(\int_{0}^{|x|}su_{c}(s)\tilde{w}_n(s)\mathrm{d}s\right)\mathrm{d}x\nonumber\\
			&       \ \ \ \ +t\int_{\R^2}\nabla u_{c}\cdot \nabla \tilde{w}_n\mathrm{d}x
			-\tau^{-2}t\int_{\R^2}\left(e^{u_c^2}-1\right)u_c\tilde{w}_n\mathrm{d}x\nonumber\\
			&       \ \ \ \  +t^2\left[O\left(\f{1}{\log n}\right)\right]
			+t^3\left[O\left(\f{1}{\log^{3/2} n}\right)\right]+t^4\left[O\left(\f{1}{\log^{2} n}\right)\right]\nonumber\\
&  =   \Phi(u_{c})+\Psi_n(t)+ \langle \Phi'(u_{c}), t \tilde{w}_n\rangle
			-(1-\tau^{-4})\left\{\frac{1}{2}B(u_{c})\right.\nonumber\\
			&      \ \ \ \ +t\left.\int_{\mathbb{R}^{2}}\left[\frac{u_c\tilde{w}_n}{|x|^{2}}
			\left(\int_{0}^{|x|}\frac{s}{2}u_{c}^{2}(s)\mathrm{d}s\right)^2+\frac{u_{c}^{2}}{|x|^{2}}
			\left(\int_{0}^{|x|}\frac{s}{2}u_{c}^{2}(s)\mathrm{d}s\right)
			\left(\int_{0}^{|x|}su_{c}(s)\tilde{w}_n(s)\mathrm{d}s\right)\right]\mathrm{d}x\right\}\nonumber\\
			&      \ \ \ \  +\frac{1-\tau^{-2}}{2}\int_{\R^2}\left(e^{u_{c}^2}-1-u_{c}^2\right)\mathrm{d}x
			+\left(1-\tau^{-2}\right)t\int_{\R^2}
			\left(e^{u_c^2}-1\right)u_c \tilde{w}_n \mathrm{d}x\nonumber\\
& \ \ \ \  +  \int_{\R^2}g(x)(u_c+t \tilde{w}_n)\mathrm{d}x-\int_{\R^2}g(x)\left[u_c(\tau x)+t \tilde{w}_n(\tau x)\right]\mathrm{d}x \nonumber\\
			&       \ \ \ \  +t^2\left[O\left(\f{1}{\log n}\right)\right]
			+t^3\left[O\left(\f{1}{\log^{3/2} n}\right)\right]+t^4\left[O\left(\f{1}{\log^{2} n}\right)\right]\nonumber\\
&  =   m(c)+\Psi_n(t)-\lambda_ct\int_{\R^2}u_c\tilde{w}_n\mathrm{d}x\nonumber\\
			&      \ \ \ \ -\left\{1-\left[1+\frac{2t}{c}\int_{\R^2}u_c\tilde{w}_n\mathrm{d}x
			+t^2\left(O\left(\frac{1}{\log n}\right)\right)\right]^{-2}\right\}\left\{\frac{1}{2}B(u_{c})\right.\nonumber\\
			&      \ \ \ \ +t\left.\int_{\mathbb{R}^{2}}\left[\frac{u_c\tilde{w}_n}{|x|^{2}}
			\left(\int_{0}^{|x|}\frac{s}{2}u_{c}^{2}(s)\mathrm{d}s\right)^2+\frac{u_{c}^{2}}{|x|^{2}}
			\left(\int_{0}^{|x|}\frac{s}{2}u_{c}^{2}(s)\mathrm{d}s\right)
			\left(\int_{0}^{|x|}su_{c}(s)\tilde{w}_n(s)\mathrm{d}s\right)\right]\mathrm{d}x\right\}\nonumber\\
			&      \ \ \ \ +\frac{1}{2}\left\{1-\left[1+\frac{2t}{c}\int_{\R^2}u_c\tilde{w}_n\mathrm{d}x
			+t^2\left(O\left(\frac{1}{\log n}\right)\right)\right]^{-1}\right\}
			\int_{\R^2}\left(e^{u_{c}^2}-1-u_{c}^2\right)\mathrm{d}x\nonumber\\
			&      \ \ \ \ +\left\{1-\left[1+\frac{2t}{c}\int_{\R^2}u_c\tilde{w}_n\mathrm{d}x
			+t^2\left(O\left(\frac{1}{\log n}\right)\right)\right]^{-1}\right\}t\int_{\R^2}
			\left(e^{u_c^2}-1\right)u_c\tilde{w}_n\mathrm{d}x\nonumber\\
& \ \ \ \  +  \int_{\R^2}g(x)(u_c+t\tilde{w}_n)\mathrm{d}x-\int_{\R^2}g(x)\left[u_c(\tau x)+t\tilde{w}_n(\tau x)\right]\mathrm{d}x \nonumber\\
			&       \ \ \ \  +t^2\left[O\left(\f{1}{\log n}\right)\right]
			+t^3\left[O\left(\f{1}{\log^{3/2} n}\right)\right]+t^4\left[O\left(\f{1}{\log^{2} n}\right)\right]\nonumber\\
& \le   m(c)+\Psi_n(t)-\lambda_ct\int_{\R^2}u_c\tilde{w}_n\mathrm{d}x\nonumber\\
			&      \ \ \ \ -\left[\f{4t}{c}\int_{\R^2}u_c\tilde{w}_n\mathrm{d}x
			+t^2\left(O\left(\f{1}{\log n}\right)\right)\right]\left\{\frac{1}{2}B(u_{c})\right.\nonumber\\
			&      \ \ \ \ +t\left.\int_{\mathbb{R}^{2}}\left[\frac{u_c\tilde{w}_n}{|x|^{2}}
			\left(\int_{0}^{|x|}\frac{s}{2}u_{c}^{2}(s)\mathrm{d}s\right)^2+\frac{u_{c}^{2}}{|x|^{2}}
			\left(\int_{0}^{|x|}\frac{s}{2}u_{c}^{2}(s)\mathrm{d}s\right)
			\left(\int_{0}^{|x|}su_{c}(s)\tilde{w}_n(s)\mathrm{d}s\right)\right]\mathrm{d}x\right\}\nonumber\\
			&      \ \ \ \ +\left[\frac{t}{c}\int_{\R^2}u_c\tilde{w}_n\mathrm{d}x
			+t^2\left(O\left(\frac{1}{\log n}\right)\right)\right]
			\int_{\R^2}\left(e^{u_{c}^2}-1-u_{c}^2\right)\mathrm{d}x\nonumber\\
			&      \ \ \ \ +\left[\frac{2t}{c}\int_{\R^2}u_c\tilde{w}_n\mathrm{d}x
			+t^2\left(O\left(\frac{1}{\log n}\right)\right)\right]t\int_{\R^2}
			\left(e^{u_c^2}-1\right)u_c\tilde{w}_n\mathrm{d}x\nonumber\\
& \ \ \ \  +  \int_{\R^2}g(x)(u_c+t\tilde{w}_n)\mathrm{d}x-\int_{\R^2}g(x)\left[u_c(\tau x)+t\tilde{w}_n(\tau x)\right]\mathrm{d}x \nonumber\\
			&       \ \ \ \  +t^2\left[O\left(\f{1}{\log n}\right)\right]
			+t^3\left[O\left(\f{1}{\log^{3/2} n}\right)\right]+t^4\left[O\left(\f{1}{\log^{2} n}\right)\right]\nonumber\\
&  =      m(c)+\Psi_n(t)
       +\frac{2t^2}{c} \left(\int_{\R^2}u_c\tilde{w}_n\mathrm{d}x\right)
\left[\int_{\R^2}\left(e^{u_c^2}-1\right)u_c\tilde{w}_n\mathrm{d}x\right]\nonumber\\
        &      \ \ \ \  -\f{4t^2}{c}\left(\int_{\R^2}u_c\tilde{w}_n\mathrm{d}x\right)
			\int_{\mathbb{R}^{2}}\left[\frac{u_{c}^{2}}{|x|^{2}}
			\left(\int_{0}^{|x|}\frac{s}{2}u_{c}^{2}(s)\mathrm{d}s\right)
			\left(\int_{0}^{|x|}su_{c}(s)\tilde{w}_n(s)\mathrm{d}s\right)\right. \nonumber\\
        &\ \ \ \ +\left.\frac{u_c\tilde{w}_n}{|x|^{2}} \left(\int_{0}^{|x|}\frac{s}{2}u_{c}^{2}(s)\mathrm{d}s\right)^2\right]\mathrm{d}x
		+\int_{\R^2}g(x)u_c\mathrm{d}x-\int_{\R^2}g(x)u_c(\tau x)\mathrm{d}x \nonumber\\
        & \ \ \ \ -\frac{t}{c}\left(\int_{\R^2}u_c\tilde{w}_n\mathrm{d}x\right)  \int_{\R^2} \left[2g(x)+
                            \nabla g(x) \cdot x \right] u_c \mathrm{d}x \nonumber\\
        &    \ \ \ \ +t^2\left[O\left(\f{1}{\log n}\right)\right]
			+t^3\left[O\left(\f{1}{\log^{3/2} n}\right)\right]+t^4\left[O\left(\f{1}{\log^{2} n}\right)\right]
		\end{align}
By {\rm(G1)}-{\rm(G4)}, \eqref{M25+}\ and\ \eqref{tau122}, we can deduce that
\begin{align}\label{in}
 &\int_{\R^2}g(x)\left( u_c-u_c(\tau x) \right)\mathrm{d}x -\frac{t}{c}\left(\int_{\R^2}u_c\tilde{w}_n\mathrm{d}x\right)  \int_{\R^2} \left[2g(x)+
\nabla g(x) \cdot x \right] u_c \mathrm{d}x \nonumber\\
  &  =\int_{\R^2}g(x)u_c\mathrm{d}x-\frac{1}{\tau^2}\int_{\R^2}g(x)u_c\mathrm{d}x+\frac{1}{\tau^2}\int_{\R^2}g(x)u_c\mathrm{d}x
  -\frac{1}{\tau^2}\int_{\R^2}g\left(\frac{x}{\tau}\right)u_c\mathrm{d}x\nonumber\\
  &\ \ \ \ -\frac{t}{c}\left(\int_{\R^2}u_c\tilde{w}_n\mathrm{d}x\right)  \int_{\R^2} \left[2g(x)+
\nabla g(x) \cdot x \right] u_c \mathrm{d}x\nonumber\\
 & \le  \frac{1}{\tau^2}\int_{\R^2}\left(g(x)-g\left(\frac{x}{\tau}\right)\right)u_c\mathrm{d}x
  -\frac{t}{c}\left(\int_{\R^2}u_c\tilde{w}_n\mathrm{d}x\right)  \int_{\R^2} \left[\nabla g(x) \cdot x \right] u_c \mathrm{d}x\nonumber\\
   &  \leq \frac{\tau-1}{\tau^3}\int_{\R^2}\left[\nabla g(x) \cdot x \right] u_c \mathrm{d}x-\frac{t}{c}\left(\int_{\R^2}u_c\tilde{w}_n\mathrm{d}x\right)  \int_{\R^2} \left[\nabla g(x) \cdot x \right] u_c \mathrm{d}x\nonumber\\
    & =   \frac{t}{{\tau}^3c}\left(\int_{\R^2}u_c\tilde{w}_n\mathrm{d}x\right)\int_{\R^2}\left[\nabla g(x) \cdot x \right] u_c \mathrm{d}x\nonumber\\
    &\ \ \ \ -\frac{t}{c}\left(\int_{\R^2}u_c\tilde{w}_n\mathrm{d}x\right)  \int_{\R^2} \left[\nabla g(x) \cdot x \right] u_c \mathrm{d}x+ t^2\left[O\left(\f{1}{\log n}\right)\right]\nonumber\\
    & =   t\left[O\left(\f{1}{\sqrt{\log n}}\right)\right]\left(1- \frac{1}{{\tau}^3}  \right)+ t^2\left[O\left(\f{1}{\log n}\right)\right] \nonumber\\
       &  =   t^2\left[O\left(\f{1}{\log n}\right)\right]+t^3\left[O\left(\f{1}{\log^{3/2} n}\right)\right].
 \end{align}
 Combining \eqref{H30} and \eqref{in}, it is easy to verify that for any $t>0$
 \begin{align}\label{H30222}
			&       \Phi(W_{n,t}) \leq
  m(c)+\Psi_n(t)+t^2\left[O\left(\f{1}{\log n}\right)\right]
			+t^3\left[O\left(\f{1}{\log^{3/2} n}\right)\right]+t^4\left[O\left(\f{1}{\log^{2} n}\right)\right].
\end{align}
		In the sequel, we assume that all inequalities hold for large $n\in \N$ without mentioning.
		We claim that
		\begin{align}\label{Psi}
			\sup_{t>0}\left[\Psi_n(t)+t^2\left(O\left(\frac{1}{\log n}\right)\right)
			+t^3\left(O\left(\frac{1}{\log^{3/2} n}\right)\right)+t^4\left(O\left(\f{1}{\log^{2} n}\right)\right)\right]< 2\pi.
		\end{align}
		Next, we will distinguish the following three cases to prove this claim.
		
		\smallskip
		\smallskip
		Case i):  \ $t\in \left[0,\sqrt{2\pi}\right]$. From \eqref{Pst}, one  has
		\begin{align}\label{H16}
			& \Psi_n(t)+t^2\left(O\left(\frac{1}{\log n}\right)\right)
			+t^3\left(O\left(\frac{1}{\log^{3/2} n}\right)\right)+t^4\left(O\left(\f{1}{\log^{2} n}\right)\right)\nonumber\\
			&  \le   \frac{t^2}{2}  +O\left(\frac{1}{\log  n}\right)
			\le \frac{3}{2} \pi.
		\end{align}
		
		\smallskip
		Case ii): \ $t\in \left[\sqrt{2\pi}, \sqrt{6\pi}\right)$. It follows from \eqref{wnt}
		and \eqref{tau1} that
		\begin{align}\label{H17}
			\frac{1}{\tau^2}\int_{\R^2}\left(e^{t^2\tilde{w}_n^2}-1-t^2\tilde{w}_n^2\right)\mathrm{d}x
			\ge  \frac{1}{2\tau^2}\int_{B_{1/n}}e^{t^2\tilde{w}_n^2}\mathrm{d}x
			\ge  \frac{1}{16n^2}e^{(2\pi)^{-1}t^2\log n},
		\end{align}
		which, together with \eqref{Pst}, yields
		\begin{align}\label{H18}
			\Psi_n(t)  \le   \frac{t^2}{2}-\frac{1}{32n^2}e^{(2\pi)^{-1}t^2\log n}
			:=   \varphi_n(t).
		\end{align}
		Let $t_n>0$ be such that $\varphi_n'(t_n)=0$. Then
		\begin{equation}\label{H19}
			1=\frac{\log n}{32\pi n^2}e^{(2\pi)^{-1}t_n^2\log n},
		\end{equation}
		which shows
		\begin{align}\label{H20}
			t_n^2 =    4\pi\left[1+\frac{\log (32\pi)-\log(\log n)}{2\log n}\right]
		\end{align}
		and
		\begin{equation}\label{H21}
			\varphi_n(t)\le \varphi_n(t_n)=\frac{t_n^2}{2}-\frac{\pi}{\log n}, \ \ \forall \ t\ge 0.
		\end{equation}
		Thus  from  \eqref{H18}, \eqref{H20} and \eqref{H21}, we have
		\begin{equation*}\label{H24}
			\Psi_n(t)\le \varphi_n(t)\le  \frac{t_n^2}{2}-\frac{\pi}{\log n} = 2\pi-\frac{\pi}{\log n}\log\frac{e\log n}{32\pi},
		\end{equation*}
		and so
		\begin{align*}
			& \Psi_n(t)+t^2\left(O\left(\frac{1}{\log n}\right)\right)
			+t^3\left(O\left(\frac{1}{\log^{3/2} n}\right)\right)+t^4\left(O\left(\f{1}{\log^{2} n}\right)\right)\nonumber\\
			&   \le  2\pi-\frac{\pi}{\log n}\log\frac{\log n}{32\pi}\ \ \hbox{for}\ \ t\in \left[\sqrt{2\pi}, \sqrt{6\pi}\right).
		\end{align*}
		
		\smallskip
		Case iii): \ $t\in \left(\sqrt{6\pi}, +\infty\right)$. By  \eqref{wnt}
		and \eqref{tau1},  we deduce that
		\begin{align}\label{H25}
			& \Psi_n(t)+t^2\left(O\left(\frac{1}{\log n}\right)\right)
			+t^3\left(O\left(\frac{1}{\log^{3/2} n}\right)\right)+t^4\left(O\left(\f{1}{\log^{2} n}\right)\right)\nonumber\\
			& \le   \frac{t^4}{8\pi} -\frac{1}{2\tau^2}\int_{\R^2}\left(e^{t^2\tilde{w}_n^2}-1-t^2\tilde{w}_n^2\right)\mathrm{d}x\nonumber\\
			& \le   \frac{t^4}{8\pi}-\frac{\pi}{4n^2\tau^2}e^{(2\pi)^{-1}t^2\log n}\nonumber\\
			& \le   \frac{t^4}{8\pi}-\frac{\pi}{4n^2(1+t+t^2)}e^{(2\pi)^{-1}t^2\log n}\nonumber\\
			& \le   \frac{9\pi}{2}-\frac{\pi}{4n^2(1+\sqrt{6\pi}+6\pi)}e^{3\log n}\le \frac{3}{2}\pi,
		\end{align}
		where we have used the fact that the function
		$$
		\phi_n(t):=\frac{t^4}{8\pi}-\frac{\pi}{4n^2(1+t+t^2)}e^{(2\pi)^{-1}t^2\log n}
		$$
		is decreasing on $t\in \left(\sqrt{6\pi}, +\infty\right)$ for large $n$. In fact,
		$$
		\phi_n'(t)=\frac{t^3}{2\pi}-\frac{\left(1+t+t^2\right)t\log n-(1+2t)\pi}{4n^2\left(1+t+t^2\right)^2}
		e^{(2\pi)^{-1}t^2\log n}.
		$$
		Let $s_n>0$ be such that $\phi_n'(s_n)=0$ for large $n$. Then
		\begin{equation*}
			2s_n^3\left(1+s_n+s_n^2\right)^2
			=\frac{\pi\left(1+s_n+s_n^2\right)s_n\log n-(1+2s_n)\pi^2}{n^2} e^{(2\pi)^{-1}s_n^2\log n},
		\end{equation*}
		which gives
		\begin{equation*}
			s_n^2
			=    4\pi\left\{1+\frac{\log \left[2s_n^3\left(1+s_n+s_n^2\right)^2\right]
				-\log\left[\pi\left(1+s_n+s_n^2\right)s_n\log n-(1+2s_n)\pi^2\right]}{2\log n}\right\}.
		\end{equation*}
		This implies that $\lim_{n\to\infty}s_n^2=4\pi$. So $\phi_n(t)$ is decreasing on $t\in \left(\sqrt{6\pi}, +\infty\right)$ for large $n$.
		
		Hence, \eqref{Psi} follows from Cases i)-iii). This means that there exists a large number $\bar{n}\in\N$ such that
		\begin{equation}\label{H33}
			\sup_{t>0}\Phi(W_{\bar{n},t})< m(c)+2\pi.
		\end{equation}
		Moreover, it follows from \eqref{tau1}, \eqref{Wn1},  \eqref{H30} and \eqref{Psi}
		that $W_{\bar{n},t}\in \mathcal{S}_c$ for all $t>0$, $W_{\bar{n},0}=u_c$ and $\Phi(W_{\bar{n},t}) <2m(c)$ for large $t>0$. Thus, there exists
		a number $\hat{t}>0$  such that
		\begin{equation}\label{tau2}
			\Phi(W_{\bar{n},\hat{t}}) <2m(c).
		\end{equation}
		Let $\gamma_{\bar{n}}(t):=W_{\bar{n},t\hat{t}}$. Then $\gamma_{\bar{n}}\in \Gamma_{c}$, where $\Gamma_{c}$ is defined by
		\eqref{Ga1}. Hence,  \eqref{Mmc} follows from \eqref{Mu2} and \eqref{H33}, and the proof is completed.
	\end{proof}

	\begin{lemma}\label{lem 3.15}
		Let $u_n\rightharpoonup \hat{u}$ in $H_r^1(\R^2)$ and
		\begin{equation*}\label{fun1}
			\int_{\R^2}\left(e^{u_n^2}-1\right)u_n^2\mathrm{d}x\le K_0
		\end{equation*}
		for some constant $K_0>0$. Then
		\begin{enumerate}[{\rm (i)}]
			\item $\lim_{n\to\infty}\int_{\R^2}\left(e^{u_n^2}-1\right)u_n\phi\mathrm{d}x
			=\int_{\R^2}\left(e^{\hat{u}^2}-1\right)\hat{u}\phi\mathrm{d}x $ for all $\phi\in \mathcal{C}_{0}^{\infty}(\R^2)${\rm;}
			
			\item $\lim_{n\to\infty}\int_{\R^2}\left(e^{u_n^2}-1-u_n^2\right)\mathrm{d}x
			=\int_{\R^2}\left(e^{\hat{u}^2}-1-\hat{u}^2\right)\mathrm{d}x$.
		\end{enumerate}
	\end{lemma}
	
	\begin{proof}  (i) is a direct consequence of \cite[Lemma 2.1]{{dMR1}}. (ii) follows from \cite[Assertion 2]{CTjde1}.
	\end{proof}

	\begin{proof} [Proof of  Theorem {\rm\ref{thm 1.2}}] \  Using Lemmas \ref{lem 3.9} and
		\ref{lem 3.14}, we can derive that for any $c>0$, there exists a nonnegative sequence $\{u_n\}\subset \mathcal{S}_c$ such that
		\begin{equation}\label{JCi}
			\Phi(u_n)\rightarrow M(c)\in (0, m(c)+2\pi), \ \   \Phi|_{\mathcal{S}_c}'(u_n) \rightarrow 0\ \
			\mbox{and}\ \ \mathcal{P}(u_n)\rightarrow 0.
		\end{equation}
From {\rm(G1)}, {\rm(G2)}, \eqref{fai}, \eqref{Pu}, \eqref{GNs}, \eqref{NB0} and \eqref{JCi}, we have
		\begin{align}\label{TM25}
			&M(c)+o(1)\nonumber\\
			&  =  \Phi(u_n)-  \frac{1}{4}\mathcal{P}(u_n)\nonumber\\
			&  =   \frac{1}{4}\|\nabla u_n\|_2^2+\frac{1}{4}B(u_n)
			+\frac{1}{4}\int_{\R^2}\left[\left(u_n^2-3\right)e^{u_n^2}+3+2u_n^2\right]\mathrm{d}x-\frac{1}{4}\int_{\R^2}\left[5g(x)+\nabla g(x)\cdot x\right]u_n\mathrm{d}x\nonumber\\
			&  =   \frac{1}{4}\|\nabla u_n\|_2^2+\frac{1}{4}B(u_n)-\frac{1}{8}\|u_n\|_4^4
			+\frac{1}{4}\sum_{k=4}^{\infty}\frac{k-3}{k!}\int_{\R^2}u_n^{2k}\mathrm{d}x-\frac{1}{4}\int_{\R^2}\left[5g(x)+\nabla g(x)\cdot x\right]u_n\mathrm{d}x\nonumber\\
			& \ge  \frac{1}{4}\left(1-\frac{c}{4\pi}\right)\|\nabla u_n\|_2^2
			-\frac{5}{4}\|g(x)\|_{\frac{4}{3}}\|u_n\|_{4}-\frac{1}{4}\|\nabla g(x) \cdot x\|_{\frac{4}{3}}\|u_n\|_4\nonumber\\
& \ge \frac{1}{4}\left(1-\frac{c}{4\pi}\right)\|\nabla u_n\|_2^2
			-\frac{1}{4}c^{\frac{1}{4}}\mathcal{C}_4\left[ 5\|g(x)\|_{\frac{4}{3}}+ \|\nabla g(x) \cdot x\|_{\frac{4}{3}}\right]\|\nabla u_n\|_2^{\frac{1}{2}}
		\end{align}
		which, together with $0<c<c_0\le 2\pi$, implies $\{u_n\}$ is bounded in $H_{r}^1(\R^2)$.  Note that \eqref{JCi} and Lemma \ref{lem 2.16} give that
		\begin{equation}\label{TM29}
		\Phi'(u_n)+\lambda_nu_n\rightarrow 0,
		\end{equation}
		where $\lambda_n$ is defined by \eqref{G430}. It is easy to see that $\{|\lambda_n|\}$ is
		bounded.
		Thus, we may thus assume,  passing to a subsequence if necessary, that
		\begin{align}\label{UN}
 \left\{
\begin{aligned}
&\lambda_n \to \hat{\lambda}_c, \ \|\nabla u_n\|_2^2 \to \hat{A}^2; \\
&u_n \rightharpoonup \hat{u}, \  in  \ H^1_{r}(\mathbb{R}^2); \\
&u_n \to \hat{u}, \ in \   L^s(\mathbb{R}^2), \forall s >2; \\
&u_n \to \hat{u}, \ a.e. \ on \  \mathbb{R}^2.
\end{aligned}
\right.
\end{align}	
		Using {\rm(G1)}, {\rm(G2)}, \eqref{fai}, \eqref{Pu} and \eqref{JCi} again, we have
		\begin{align}\label{TM260}
			& M(c)+o(1)\nonumber\\
			&  =  \Phi(u_n)-  \frac{1}{2}\mathcal{P}(u_n)\nonumber\\
			&  =   \frac{1}{2}\int_{\R^2}\left[\left(u_n^2-1\right)e^{u_n^2}+1-\left(e^{u_n^2}-1-u_n^2\right)\right]\mathrm{d}x -\frac{1}{2}\int_{\R^2}\left[3g(x)+  \nabla g(x) \cdot x\right]   u_n\mathrm{d}x \nonumber\\
			&  \ge \frac{1}{6} \sum_{k=2}^{\infty}\frac{1}{k!}\int_{\R^2}u_n^{2(k+1)}\mathrm{d}x-\frac{3}{2}\|g(x)\|_{\frac{4}{3}}\|u_n\|_{4}-\frac{1}{2}\|\nabla g(x) \cdot x\|_{\frac{4}{3}}\|u_n\|_4\nonumber\\
			&  =   \frac{1}{6} \int_{\R^2}\left(e^{u_n^2}-1\right)u_n^2\mathrm{d}x- \frac{1}{6}\int_{\R^2}u_n^4\mathrm{d}x
			-\frac{3}{2}\|g(x)\|_{\frac{4}{3}}\|u_n\|_{4}-\frac{1}{2}\|\nabla g(x) \cdot x \|_{\frac{4}{3}}\|u_n\|_4,
		\end{align}
		which, together with the boundedness of $\{\|u_n\|\}$, yields
		\begin{align}\label{TM27}
			\int_{\R^2}\left(e^{u_n^2}-1\right)u_n^2\mathrm{d}x\le C_1
		\end{align}
		for some $C_1>0$.  By \eqref{TM27} and Lemma \ref{lem 3.15}, we have
		\begin{equation}\label{M32}
			\lim_{n\to\infty}\int_{\R^2}\left(e^{u_n^2}-1-u_n^2\right)u_n\phi\mathrm{d}x
			= \int_{\R^2}\left(e^{\hat{u}^2}-1-\hat{u}^2\right)\hat{u}\phi\mathrm{d}x,
			\ \ \ \ \forall \ \phi\in \mathcal{C}_0^{\infty}(\R^2)
		\end{equation}
		and
		\begin{equation}\label{M36}
			\lim_{n\to\infty}\int_{\R^2}\left(e^{u_n^2}-1-u_n^2\right)\mathrm{d}x
			=\int_{\R^2}\left(e^{\hat{u}^2}-1-\hat{u}^2\right)\mathrm{d}x.
		\end{equation}
		Since $\lambda_n\rightarrow \hat{\lambda}_c$, it follows from \eqref{UN}, \eqref{TM29} and \eqref{M32}  that
		\begin{equation}\label{M34}
			\Phi'(\hat{u})+\hat{\la}_c\hat{u}=0.
		\end{equation}
		To prove that $\hat{u}\in H^1_r(\R^2)$ is a solution to \eqref{Pa1},
		it suffices to show that $\|\hat{u}\|_2^2=c$  by \eqref{M34}. For this, we prove below three assertions in turn.
		
		\vskip2mm
		{\bf Assertion 1.} $\hat{u}\ne 0$.
		
		\vskip2mm
		Otherwise, we suppose that $u_n\rightharpoonup 0$ in $H_{r}^1(\R^2)$.
		By {\rm(G1)}, \eqref{JCi}, \eqref{UN} and \eqref{M36}, we can deduce that
		\begin{equation*}
			\|\nabla u_n\|^2
			\le  \|\nabla u_n\|^2+B(u_n)=   2\Phi(u_n)+\int_{\R^2}\left(e^{u_n^2}-1-u_n^2\right)\mathrm{d}x +2\int_{\R^2}g(x)u_n\mathrm{d}x =   2M(c)+o(1),
		\end{equation*}
		which, together with Lemma \ref{lem 3.14}, implies
		\begin{align}\label{M38}
			4\pi(1-3\bar{\varepsilon}) := \liminf_{n\to \infty}\|\nabla u_n\|^2\le 2M(c) < 2m(c)+4\pi
		\end{align}
		for some constant $\bar{\varepsilon}>0$. Let $q=(1-\bar{\varepsilon})/(1-3\bar{\varepsilon})$. Then
		$q>1$ and $q/(q-1)>1$.
		Note that \eqref{M38} and ii) of Lemma \ref{lem 2.3} yield that
		\begin{align}\label{M42}
			\int_{\R^2}\left(e^{u_n^2}-1\right)^q\mathrm{d}x
			& \le \int_{\R^2}\left(e^{q u_n^2}-1\right)\mathrm{d}x\le C_{2}.
		\end{align}
		Taking into account that $q/(q-1)>1$ and $u_n\to 0$ in $L^s(\R^2)$ for $s>2$, by \eqref{M42} and
		the H\"{o}lder inequality, we have
		\begin{equation}\label{M44}
			\int_{\R^2}\left(e^{u_n^2}-1\right)u_n^2\mathrm{d}x
			\le \left[\int_{\R^2}\left(e^{u_n^2}-1\right)^q\mathrm{d}x\right]^{1/q}\|u_n\|_{2q/(q-1)}^2
			=o(1).
		\end{equation}
		Then it follows from {\rm(G1)}, {\rm(G2)}, \eqref{TM260}, \eqref{M36} and \eqref{M44} that
		\begin{align}\label{Th8}
			&M(c)+o(1)\nonumber\\
&=\Phi(u_n)-  \frac{1}{2}\mathcal{P}(u_n)\nonumber\\
			&   = \frac{1}{2}\int_{\R^2}\left[\left(e^{u_n^2}-1\right)u_n^2
			-2\left(e^{u_n^2}-1-u_n^2\right)\right]\mathrm{d}x-\frac{1}{2}\int_{\R^2}\left[3g(x)+  \nabla g(x) \cdot x\right]   u_n\mathrm{d}x\nonumber\\
			&   =   o(1),
		\end{align}
		which contradicts with the fact that $M(c)>0$ for any $c\in (0,c_0)$. This contradiction shows that $\hat{u}\ne 0$.
		
		\vskip2mm
		{\bf Assertion 2.} $\int_{\R^2}\left(e^{u_n^2}-1\right)u_n(u_n-\hat{u})\mathrm{d}x=o(1)$.
		
		\smallskip
		Since $\hat{u}\ne 0$, it follows \eqref{M34} and Lemma \ref{Lep} that $\mathcal{P}(\hat{u})=0$. Arguing as in the proof of \eqref{G46} and \eqref{lamd}, we have $\hat{\la}_c > 0$ and  further derive that $\|u_n\|_2^2=\|\hat{u}\|_2^2$.
		It follows from \eqref{fai}, \eqref{JCi}, \eqref{UN}, \eqref{M36}, Lemma \ref{lem2.1} and {\rm(G1)} that
\begin{align}\label{fm11}
	&M(c)+o(1)=\Phi(u_n)\nonumber\\
		  &=   \frac{1}{2}\|\nabla u_n\|_2^2+ \frac{1}{2}B(u_n)
			-\frac{1}{2}\int_{\R^2}\left(e^{u_n^2}-1-u_n^2\right)\mathrm{d}x-\int_{\R^2}g(x)u_n\mathrm{d}x\nonumber\\
			&  =   \frac{1}{2}\|\nabla u_n\|_2^2+\|\nabla \hat{u}\|_2^2+ \frac{1}{2}B(\hat{u})
			-\frac{1}{2}\int_{\R^2}\left(e^{\hat{u}^2}-1-\hat{u}^2\right)\mathrm{d}x-\int_{\R^2}g(x)\hat{u}\mathrm{d}x+o(1)\nonumber\\
			&  =   \frac{1}{2}\|\nabla u_n\|_2^2-\frac{1}{2}\|\nabla \hat{u}\|_2^2+\Phi(\hat{u})+o(1)\nonumber\\
&  =   \frac{1}{2}\hat{A}^2-\frac{1}{2}\|\nabla \hat{u}\|_2^2+\Phi(\hat{u}).
	\end{align}
According to \eqref{fm11}, Lemmas \ref{lem 3.3} and \ref{lem 3.8},  one can easily that $\Phi(\hat{u})\ge m(c)$. Therefore
	\begin{align}\label{Th11}
	&M(c)+o(1)  =   \frac{1}{2}\|\nabla (u_n- \hat{u})\|_2^2+\Phi(\hat{u})+o(1)\ge  \frac{1}{2}\|\nabla (u_n- \hat{u})\|_2^2+m(c)+o(1).
	\end{align}
		Since $0<M(c)<m(c)+2\pi$ for any $c\in (0,c_0)$, we know from \eqref{Th11}  that
		there exists $\bar{\varepsilon}>0$ such that
		\begin{equation}\label{Th12}
			\|\nabla (u_n- \hat{u})\|_2^2\le 4\pi(1-3\bar{\varepsilon}), \ \ \ \ \mbox{for large } n\in \N.
		\end{equation}
		Choose $q\in (1,2)$ such that $q^2(1-3\bar{\varepsilon})\le (1-\bar{\varepsilon})$. Then
		by \eqref{Th12}, the Young's inequality and Lemma \ref{lem 2.3} ii), we can deduce that
		\begin{align}\label{Th14}
			&      \int_{\R^2}\left(e^{u_n^2}-1\right)^q\mathrm{d}x \le  \int_{\R^2}\left(e^{q u_n^2}-1\right)\mathrm{d}x\nonumber\\
			& \le  \int_{\R^2}\left[e^{(1+\bar{\varepsilon}^{-1})q \hat{u}^2}
			e^{(1+\bar{\varepsilon})q (u_n-\hat{u})^2}-1\right]\mathrm{d}x\nonumber\\
			& \le  \frac{q-1}{q}\int_{\R^2}\left[e^{(1+\bar{\varepsilon}^{-1})
				q^2(q-1)^{-1} \hat{u}^2}-1\right]\mathrm{d}x+\frac{1}{q}\int_{\R^2}\left[e^{(1+\bar{\varepsilon})q^2 (u_n-\hat{u})^2}-1\right]
			\mathrm{d}x \le  C_3.
		\end{align}
		Noting that $q/(q-1)>1$, by \eqref{UN}, \eqref{Th14} and the H\"{o}lder inequality, we have
		\begin{align}\label{Th15}
			&       \int_{\R^2}\left(e^{u_n^2}-1\right)u_n(u_n-\hat{u})\mathrm{d}x\nonumber\\
			& \le   \left[\int_{\R^2}\left(e^{u_n^2}-1\right)^q\mathrm{d}x\right]^{1/q}\|u_n\|_{2q/(q-1)}
			\|u_n-\hat{u}\|_{2q/(q-1)} = o(1).
		\end{align}
		Hence, Assertion 2 is proved.
		
		\smallskip
		{\bf Assertion 3.} $u_n\to \hat{u}$ in $H_{r}^1(\R^2)$.
		
		\smallskip
It follows from {\rm(G1)}, \eqref{G47}, \eqref{UN}, \eqref{TM29}, \eqref{M34} and Assertion 2
		that
		\begin{align*}
			o(1)
			& =   \left\langle  \Phi'({u}_n)+\la_n {u}_n,{u}_n-\hat{u}\right\rangle  \nonumber\\
			&   =   \|\nabla({u}_n-\bar{u})\|_2^2+\hat{\la}_c\|{u}_n-\hat{u}\|_2^2+o(1),
		\end{align*}
		which implies that $u_n\to \hat{u}$ in $H_{r}^1(\R^2)$. Thus, Assertion 3 holds.  This completes the proof of Theorem \ref{thm 1.2}.
	\end{proof}

\section*{Acknowledgments}
This work is partially supported by the National Natural Science Foundation of China (No: 12371181, No: 12471175),
Hunan Provincial Natural Science Foundation
(No: 2022JJ20048, No: 2021JJ40703).

\end{document}